\documentclass[11]{amsart}
\usepackage{amssymb}
\usepackage{amscd}
\usepackage{amsfonts}
\usepackage{amsmath}
\usepackage{amsthm}
\usepackage{amsxtra}
\usepackage{graphicx}
\usepackage[all]{xy}
\usepackage{verbatim}
\usepackage{epsfig}
\usepackage{color}
\usepackage{epsfig}

\theoremstyle{plain}
\newtheorem{prop}{Proposition}[section]
\newtheorem{thm}[prop]{Theorem}
\newtheorem{cor}[prop]{Corollary}

\newtheorem{conj}[prop]{Conjecture}

\theoremstyle{definition}
\newtheorem{defn}[prop]{Definition}
\newtheorem{example}[prop]{Example}
\newtheorem{rem}[prop]{Remark}

\numberwithin{equation}{section}

\DeclareMathOperator{\Hilb}{Hilb}

\DeclareMathOperator{\tr}{tr}
\DeclareMathOperator{\Proj}{Proj}
\DeclareMathOperator{\Spec}{Spec}

\DeclareMathOperator{\Hom}{Hom}

\DeclareMathOperator{\Tor}{Tor}
\newcommand{\scprod}[2]{\langle #1, #2 \rangle}

\newcommand{\C}{\mathbb{C}}
\newcommand{\N}{\mathbb{N}}
\newcommand{\Z}{\mathbb{Z}}
\newcommand{\PP}{\mathbb{P}}
\newcommand{\Q}{\mathbb{Q}}

\newcommand{\TT}{\mathbb{T}}

\newcommand{\Hs}[1]{\mathcal{H}_{#1}}

\newcommand{\shf}[1]{\mathcal{O}_{#1}}
\newcommand{\Dy}[1]{\mathcal{D}_{#1}}
\newcommand{\OO}[1]{\mathcal{O}_{#1}}
\newcommand{\llbracket}{{[\![}}
\newcommand{\rrbracket}{{]\!]}}
\newcommand{\xx}{{\mathbf x}}
\newcommand{\yy}{{\mathbf y}}

\newcommand{\aff}{\mathbb{A}^2}
\newcommand{\Hil}[1]{\Hilb^{#1}(\aff)}

\newcommand{\y}[1]{\mathcal{Y}_{#1}}
\newcommand{\mac}[1]{\widetilde{H}_{#1}(X,q,t)}
\newcommand{\fro}[2]{\mathcal{F}_{#1}(#2)}
\newcommand{\row}{\mathcal{R}ow}
\newcommand{\col}{\mathcal{C}ol}
\newcommand{\basis}[1]{\mathfrak{B}_{#1}}
\newcommand{\canonical}[1]{\mathfrak{C}_{#1}}

\begin{document}

%%%%%%%%%%%%%%%%%%%%%%%%%%%%%%%%%%%%%%%%%%%%%%%%%%%%%%%%%%%%%%%%%%%%%%%%

\title{Nested Hilbert schemes and the nested q,t-Catalan series}

\author{Mahir Bilen Can}
\address{University of Western Ontario, Canada}
\email{mcan@uwo.ca}


\date{}


\begin{abstract}
In this paper we study the tangent spaces of the smooth nested Hilbert scheme $\Hil{n,n-1}$ of points in the plane, and give a general formula for computing the Euler characteristic of a $\TT^2$-equivariant locally free sheaf on $\Hil{n,n-1}$. Applying our result to a particular sheaf, we conjecture that the result is a polynomial in the variables $q$ and $t$ with non-negative integer coefficients. We call this conjecturally positive polynomial as \textsl{the nested $q,t$-Catalan series}, for it has many conjectural properties similar to that of the $q,t$-Catalan series.
\end{abstract}

\maketitle

%%%%%%%%%%%%%%%%%%%%%%%%%%%%%%%%%%%%%%%%%%%%%%%%%%%%%%%

\section{\textbf{Introduction}}\label{intro}
Let $\mathfrak{S}_n$ be the symmetric group. The Frobenius character $\mathcal{F}$ is a map from the Grothendieck group $Rep(\mathfrak{S}_n)$ of representations of $\mathfrak{S}_n$ into the ring of symmetric functions that sends an irreducible representation $V^{\lambda}$ to the Schur function $s_{\lambda}$. 
Among the $\mathfrak{S}_n$-modules which has a (very) interesting image under the Frobenius map is the ring of diagonal coinvariants which is defined as follows. 

The symmetric group $\mathfrak{S}_n$ acts (diagonally) on the polynomial ring $\C[\mathbf{x},\mathbf{y}]:=\C[x_1,y_1,...,x_n,y_n]$ by $\sigma(x_i)=x_j,\ \sigma(y_i)=y_j$. The ring of diagonal coinvariants $R_n=\C[\mathbf{x},\mathbf{y}]/I_+$ is the quotient of $\C[\mathbf{x},\mathbf{y}]$ by the ideal $I_+$ generated by the invariant polynomials with no constant term. This is a bigraded ring (by degree), 
\begin{equation*}
R_n= \bigoplus_{r,s} (R_n)_{r,s},
\end{equation*}
and the action of $\mathfrak{S}_n$ respects the bigrading, and hence the two variable Frobenius series
\begin{equation*}
\fro{R_n}{q,t} = \sum_{r,s} \mathcal{F}((R_n)_{r,s}) q^r t^s
\end{equation*}
makes sense. Here $\mathcal{F}((R_n)_{r,s}) $ is the ordinary Frobenius character of the $\mathfrak{S}_n$-module $(R_n)_{r,s}$. 

There is a compact way of writing $\fro{R_n}{q,t}$, however, it requires 
(modified) Macdonald polynomials $\{ \mac{\mu} \}_{\mu \in \y{m}}$, which constitutes a vector space basis for the ring of symmetric functions $\Lambda^n_{\Q(q,t)}$. The existence of these polynommials is proved by I. Macdonald (see \cite{Mac95} for details).  A closed formula for an arbitrary $\mac{\mu}$ has been conjectured by Haglund \cite{Hag04} and later proved by Haglund, Haiman and Loehr \cite{HHL05}. 

The Bergeron-Garsia operator $\nabla: \Lambda^n_{\Q(q,t)} \rightarrow \Lambda^n_{\Q(q,t)}$ is defined by setting
\begin{equation*}
\nabla \mac{\mu} = t^{n(\mu )}q^{n(\mu ')} \mac{\mu},
\end{equation*}
where $n(\mu) = \sum (i-1)\mu_i$ and $\mu'$ is the conjugate partiton to 
$\mu$. 

The following highly nontrivial result about the Frobenius character (or series) of $R_n$ has been conjectured by Garsia and Haiman in \cite{GaHa96} and finally been proved by Haiman in \cite{Hai02}. 

\begin{thm}
The Frobenius character $\fro{R_n}{q,t}$ of the ring of diagonal coinvariants is equal to $\nabla e_n$, where $e_n$ is the $n^{th}$ elementary symmetric function. 
\end{thm}

The proof of this theorem involves deep geometric properties of the Hilbert scheme of $n$ points in the plane, and nontrivial manipulations involving Macdonald polynomials. The reader who wishes to see a brief survey about the proof may wish to check \cite{procesi}. We should mention here that I. Gordon, \cite{Gordon03} using rational Cherednik algebras has verified conjectures of Haiman concerning Frobenius characters of the diagonal coinvariants for other finite Coxeter groups.

In this paper, we shall be concerned with the nested Hilbert scheme $\Hil{n,n-1}$. This is the parametrizing scheme of pairs $(Z_1,Z_2)$ of 0-dimensional subschemes in the plane such that $Z_1$ is a subscheme of $Z_2$. We shall employ the Atiyah-Bott Lefschetz formula to calculate (equivariant) Euler characteristics of vector bundles on the nested Hilbert scheme. Along the way, we shall present some combinatorial conjectures related to certain Euler characteristics.

\section{\textbf{Summary of the results}}

\subsection{Combinatorics} We shall consider the subspace $R_n^{\varepsilon}$ of $\mathfrak{S}_n$-alternating polynomials in $R_n$. This is the vector subspace of $R_n$ generated by the images of those polynomials $f\in A$ with 
\begin{equation*}
\sigma \cdot f=(-1)^{sign(\sigma)}f.
\end{equation*} 
There is a natural bigrading (by degree) on $R_n^{\varepsilon}$. So we write $R_n^{\varepsilon}= \bigoplus_{r,s} (R_n^{\varepsilon})_{r,s}$ and consider the Hilbert series 
\begin{equation*}
\Hs{R_n^{\varepsilon}}(q,t) = \sum_{r,s} \dim(R_n^{\varepsilon})_{r,s}q^r t^s \in \Z[q,t].
\end{equation*}
In \cite{Hai94}, it has been (empirically) observed by Haiman that the 
specialization $\Hs{R_n^{\varepsilon}}(q,t)$ at $q=t=1$ gives
the Catalan numbers
\begin{equation*}
\Hs{R_n^{\varepsilon}}(1,1)=\frac{1}{n+1} {2n\choose n}.
\end{equation*}
It is well known that the Catalan numbers count the ``Dyck paths,'' which are the lattice paths in the $xy$-plane, starting at $(0,0)$ and ending at $(n,n)$ with unit upward ($(0,1)-step$) and unit rightward ($(1,0)-step)$ segments, and staying weakly above the diagonal $x=y$ (see figure \ref{F:dyck8} below). Let $\Dy{n}$ denote the set of all Dyck paths of size $n$. Notice that a path $P\in \Dy{n}$ can be ended by a sequence 
$seq(P)=(a_1 a_2 \cdots a_{2n})$ of 0's and 1's, where an occurrence of 0 represents an upward step and an occurrence of 1 represents a rightward step. For example, the Dyck path in figure \ref{F:dyck8} below is represented by the sequence $seq(P)=(0010011101001011)$.

\begin{figure}[h]
\begin{center}
\scalebox{.5}{\input{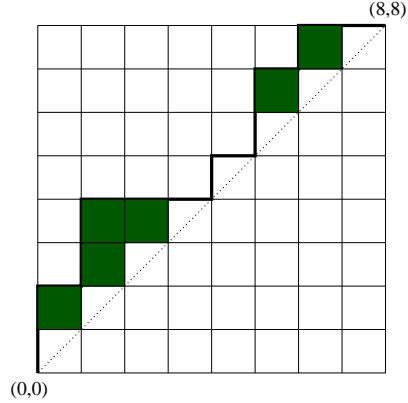}}
\caption{A Dyck path of size 8 with $area(P)=6$.}
\label{F:dyck8}
\end{center}
\end{figure}

One can ask the following natural question: could we find a pair of functions $s_1,s_2: \Dy{n} \rightarrow \N$ such that 
\begin{equation*}
\Hs{R_n^{\varepsilon}}(q,t) = \sum_{P\in \Dy{n}} q^{s_1(P)}t^{s_2(P)}.
\end{equation*}

An answer to this question has been conjectured by Haglund \cite{Hag03}, and proved later by Garsia and Haglund \cite{GaHag02}. For a basic introduction to the techniques of the proof and more, we recommend the recent book \cite{Hag07}.

Since it will be useful for our purposes as well, we shall describe the functions $s_1$ and $s_2$. 

The first function $s_1$ is, in a sense very classical; given a Dyck path $P\in \Dy{n}$, $s_1(P)$ is defined to be the ``$area$'' of $P$. This is the number of full cells below the path and above the main diagonal. It is not hard to see that $s_1(P)=area(P)$ is equivalent to the $coinv(seq(P))-{n+1 \choose 2},$ where $coinv$ of a sequence $(a_1\cdots a_{2n})$ is the number of pairs of indices $(i,j)$ such that $i<j$ and $a_i < a_j$.

The value $s_2(P)$ at $P\in \Dy{n}$ of the second function $s_2$ is called the $bounce$ number of the path $P\in \Dy{n}$. It is defined algorithmically as follows.

\textsl{Given $P\in \Dy{n}$, starting at $(n,n)$ we move leftward (in the negative $x$-axis direction) until encountering a lattice point $(j_1,n)$ such that the line segment $\overline{(j_1,n)(j_1,n-1)}$ is an upward step of $P$. Next, we start at the lattice point $(j_1,j_1)$ on the diagonal and move leftward  once again until encountering a lattice point $(j_2,j_1)$ such that the line segment $\overline{(j_2,j_1)(j_2,j_1-1)}$ is an upward step of $P$. Next, we start at $(j_2,j_2)$ on the diagonal and move leftward until encountering a lattice point $(j_3,j_2)$ such that the line segment $\overline{(j_3,j_2)(j_3,j_2-1)}$ is a rightward step of $P$. Once again we start at $(j_3,j_3)$ and repeat the  process. This continues until we reach to the point $(j_b,j_b)=(0,0)$ on the diagonal.}

\begin{defn}\label{D:bouncedefns}
Let $\{(0,0),(j_{b-1},j_{b-1}),...,(j_1,j_1)\}$ be the set of lattice points on the diagonal that we obtained by the above procedure. Let $b(P)$ be the Dyck path represented by the 
sequence
\begin{equation*}
seq(b(P))=(\underbrace{0\cdots0}_{j_{b-1}}\underbrace{1\cdots 1}_{j_{b-1}} \underbrace{0\cdots0}_{j_{b-2}}\underbrace{1\cdots1}_{j_{b-2}}\cdots \underbrace{1\cdots 1}_{j_1}).
\end{equation*}
\begin{itemize}
\item We shall call $b(P)$ as the $bounce\ path$ of $P$. 
\item The part of $P$ between the lattice points $(j_i,j_{i-1})$ and $(j_{i+1},j_i)$ is called the $i'th\ bounce\ section$ of $P$, and denoted by $\mathbf{B}_i$ (see figure \ref{F:bounceregion} below). 
\item The bounce number $s_2(P)$ is defined to be the sum
\begin{equation*}
s_2(P)=\sum_{i=1}^{b-1} n-j_i.
\end{equation*}
\end{itemize}
\end{defn}

\begin{figure}[h]
\begin{center}
\scalebox{.5}{\input{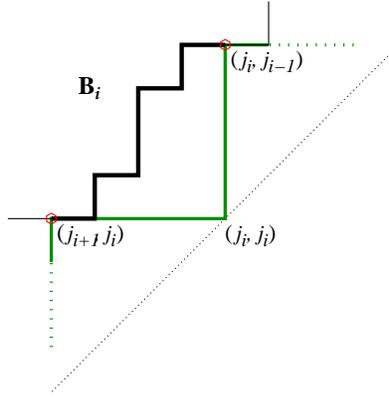}}
\caption{$i$th bounce region (between circles)}
\label{F:dyck8}
\end{center}
\end{figure}

\begin{example}
Consider the path given by the sequence $seq(P)=(0011001001101011)$, as in figure \ref{F:dyck8bounce} below. The sequence of the bounce path $b(P)$ is $seq(b(P))=(0011010001110011)$. Then, $j_1=6,j_2=3,j_3=2,j_4=0$. Therefore, the bounce number of $P$ is equal to $s_2(P)=2+5+6=13.$ 
\end{example}

Note that there exists a unique Dyck path $P_0$ with $s_2(P_0)=0$ and $s_1(P_0)= {n \choose 2}$. This the path with $n$ immediate upward steps. In other words,
\begin{equation}\label{E:biggest}
seq(P_0)= (00\cdots 011\cdots 1).
\end{equation}

\begin{figure}[h]
\begin{center}
\scalebox{.5}{\input{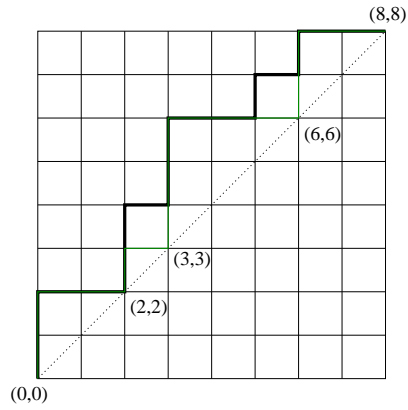}}
\caption{A Dyck path with $s_2(P)=bounce(P)=11$.}
\label{F:dyck8bounce}
\end{center}
\end{figure}

\begin{defn}
The $q,t$-Catalan series of size $n$ is defined to be the summation 
\begin{equation}
C_n(q,t)=\sum_{P\in \Dy{n}} q^{s_1(P)}t^{s_2(P)},
\end{equation}
where $s_1(P)$ and $s_2(P)$ are, respectively, the area and the bounce of the path $P$, as described above. 
\end{defn}

The stunning equality $\Hs{R_n^{\varepsilon}}(q,t) = C_n(q,t)$ is one of the good reasons to concentrate on the functions $s_1$ and $s_2$. Arguably a more stunning fact, which is due to Haiman is that the polynomial $C_n(q,t)$ can also be recovered by the Atiyah-Bott-Lefschetz formula applied to a certain locally free sheaf on the ``zero fiber'' of the Hilbert scheme of points in the plane. This is one of our main motivations for this article. Before dwelling into algebraic-geometric details, we would like to define combinatorially another family of $q,t$-polynomials similar to $C_n(q,t)$, and present a conjecture of the author and J. Haglund.

\begin{defn}\label{D:nestedcatalan}
Let $P_0$ be as in in \ref{E:biggest}.
Let $P\in \Dy{n+1}^0=\Dy{n+1}\setminus \{P_0\}$ be Dyck path defined by the sequence $seq(P)=(a_1a_2\cdots a_{2{n+1}})$.
Let $\mathbf{B}_i$ be the $i$th bounce section of $P$ as defined in \ref{D:bouncedefns}, and let $v_i=v_i(P)$ be the number of 01's (i.e., 0's immediately followed by 1's) in the part of the sequence $seq(P)$ which belongs to the bounce section $\mathbf{B}_i$. Define $s_3:\Dy{n+1}^0 \rightarrow \N$ by 
\begin{equation*}
s_3(P)=\sum_{i=1}^{b-1} j_i -1.
\end{equation*}
Finally we define the \textsl{combinatorial nested $q,t$-Catalan series} $N_n(q,t)$ to be the summation
\begin{equation}\label{E:nestedcatalan}
N_n(q,t)= \sum_{P\in \Dy{n+1}^0} q^{s_1(P)}t^{s_3(P)}(v_0+v_1 t + \cdots v_{b-1} t^{b-1}).
\end{equation}
\end{defn}

\begin{example}
The smallest example is for $n=2$. In that case, we have 4 Dyck paths of size 3 which are shown in figure \ref{F:nested2} below. It is easy to check that 
\begin{eqnarray*}
N_2(q,t) &=& q^0 t \cdot (1+t) + qt \cdot 1+ q t^0 \cdot 1 + q^2 t ^0 \cdot 1\\
 &=& q^2+q+qt+t+t^2.
\end{eqnarray*}
\end{example}

\begin{conj}
The combinatorial nested $q,t$-Catalan series $N_n(q,t)$ is symmetric in the variables $q,t$. Furthermore, we conjecture that
\begin{equation}
N_n(1,1)=\frac{n}{2} C_{n+1}^{(1)}(1,1)=\frac{n}{2(n+1)} {2(n+1) \choose n+1}.
\end{equation} 
\end{conj}

See also (\ref{E:ncatalans}) below.

\begin{figure}[h]
\begin{center}
\scalebox{.5}{\input{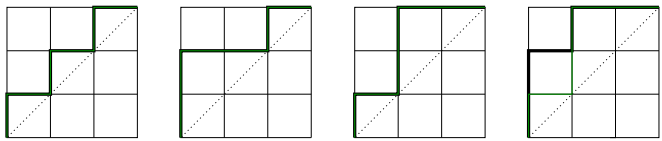}}
\caption{The elements of $\Dy{3}^0$ with their bounce paths.}
\label{F:nested2}
\end{center}
\end{figure}

\subsection{Geometry}

The Hilbert schemes of points in the plane are, arguably, among celebrities of the smooth algebraic varieties (use http://www.ams.org/mathscinet/). Since their introduction, these spaces has been well studied and understood from different perspectives. For example, in \cite{deCatMig02}, de Cataldo and Migliorini have computed the Chow motives and the Chow groups of Hilbert schemes of points. In \cite{Gordon05}, Gordon and Stafford have investigated the relationship between rational Cherednik algebras and the Hilbert schemes.
In \cite{Sam06}, Boissi\`{e}re investigates the relationship between McKay correspondence and $K$-theory of $\Hil{n}$, etc.. The literature on $\Hil{n}$ is vast, therefore, sadly we cannot avoid doing unjust towards authors whose work we can't mention here. Here are the few places that we have benefited from while trying to learn about these wonderful varieties. \cite{Iarrobino77}, \cite{EllStrom87}, \cite{Gott91}, \cite{Bertram99}, \cite{Nakajima99}, \cite{Hai01}. The reader who would like to see more on the literature may wish to check the references in \cite{Nakajima99}.

To continue with our extended introduction, we shall fix the following notation once and for all. We shall exclusively be working over the field of complex numbers. $\aff =\aff_{\C}$ shall denote the affine plane over $\C$. The Hilbert scheme of $n$ points in $\aff$ shall be denoted by $\Hil{n}$.
It can be identified with (at least set theoretically) the set of ideals $I$ in the polynomial ring $\C[x,y]$ such that the quotient ring $\C[x,y]/I$, as a vector space is of dimension $n$. Then, the nested Hilbert scheme $\Hil{n,n-1}$ can be identified with the pairs of ideals 
\begin{equation*}
(I_1,I_2)\in \Hil{n} \times \Hil{n-1},
\end{equation*}  
such that $I_1 \subseteq I_2$. We should mention here that both $\Hil{n}$ and $\Hil{n,n-1}$ are nonsingular and irreducible.

As $GL_2$ acts on $\aff$, so does the maximal torus of the diagonal invertible matrices. We shall identify $\TT^2$ with $\C^* \times \C^*$. Its action on $\aff$ passes onto $\Hil{n}$ and $\Hil{n,n-1}$.

The set of fixed points of the torus $\TT^2$ on $\Hil{n}$ is indexed by the partitions $\mu$ of $n$. 
Let $\y{n}$ be the set of partitions of the positive integer $n$. The fixed point set on $\Hil{n,n-1}$ is indexed by the pairs of partitions $(\mu,\nu)\in \y{n}\times \y{n-1}$ such that the Young diagram of $\nu$ is contained in that of $\mu$. We shall denote the set of all such pairs of partitions by $\y{n,n-1}$. Let $a(\zeta)$ and $l(\zeta)$ denote the arm and leg of the ``corner cell'' $\zeta\in \mu$ that we need to take off to get $\nu$. The notations $\row(\zeta)$ and $\col(\zeta)$ stand for the row and column of the cell $\zeta\in \mu$ in the diagram of the partition $\mu$. We shall explain these definitions in more detail in the next section.

Now we can state one of our main results. 

\begin{thm}\label{T:nestedABL}
Let $M$ be a $\TT^2$-equivariant vector bundle on $\Hil{n,n-1}$, and $\Hs{M(I_{\mu},I_{\nu})}(q,t)$ denote the Hilbert series of the fiber $M(I_{\mu}, I_{\nu})$ of $M$ at the torus fixed point $(I_{\mu},I_{\nu})\in \Hil{n,n-1}$, where $\nu$ is a partition induced from $\mu$ by taking off a corner cell $\zeta \in \mu$ (we write $\nu = \mu \setminus \{\zeta\}$). 
Then, the equivariant Euler characteristic of $M$ is 
\begin{equation}\label{E:NAB}
\chi _{M}(q,t) = \sum _{(\mu,\nu)\in \y{n,n-1}} \frac{\Hs{M(I_{\mu},I_{\nu})}(q,t)}{ (1-t)(1-q)P_1(\mu,\nu)P_2(\mu,\nu)P_3(\mu,\nu)},
\end{equation}
where $\nu$ differs from $\mu$ by a corner cell $\zeta\in \mu$ such that $P_1,P_2$ and $P_3$ are given by
\begin{eqnarray*}
P_1(\mu,\nu) &=& \prod_{\zeta \in \mu \setminus (
\row(\zeta)\cup \col(\zeta))} (1-t^{1+l(\zeta)}q^{-a(\zeta)})(1-t^{-l(\zeta)}q^{1+a(\zeta)}),\\
P_2(\mu,\nu) &=& \prod_{\zeta \in \row(\zeta) }(1-t^{1+l(\zeta)}q^{-a(\zeta)})(1-t^{-l(\zeta)}q^{a(\zeta)}),\\
P_3(\mu,\nu) &=& \prod_{\zeta \in \col(\zeta) }(1-t^{-l(\zeta)}q^{1+a(\zeta)})(1-t^{l(\zeta)}q^{-a(\zeta)}).
\end{eqnarray*}
\end{thm}

Let $\Lambda^n_{\Q(q,t)}$ be the space of homogeneous symmetric functions of degree $n$ on a set of algebraically independent variables $\{x_1,x_2,x_3...\}$ over the field of rational functions $\Q(q,t)$. As vector space, $\Lambda^n_{\Q(q,t)}$ has quite a few different distinguished bases. One of them is the basis $\{e_{\mu}\}_{\mu \in \y{n}}$ of elementary symmetric functions, defined as follows. Suppose $\mu\in \y{n}$ is a partition with parts $(\mu_1,...,\mu_k)$. Then, for $i=1...k$ define
\begin{equation*}
e_{\mu_i}= \sum_{1\leq j_1<\cdots < j_{\mu_i}} x_{j_1} x_{j_2}\cdots x_{j_{\mu_i}},\ \text{and}\ e_{\mu} = e_{\mu_1}\cdots e_{\mu_k}.
\end{equation*}
Another basis for $\Lambda^n_{\Q(q,t)}$ is given by the (modified) Macdonald polynomials 
$\mac{\mu}$, whose existence is proved by I. Macdonald (see \cite{Mac95} for details).  A closed formula for these polynomials has been conjectured by Haglund \cite{Hag04} and later proved by Haglund, Haiman and Loehr \cite{HHL05}. 

The Bergeron-Garsia operator $\nabla$ on $\Lambda^n_{\Q(q,t)}$ is defined by declaring Macdonald polynomials as eigenfuntions such that 
\begin{equation*}
\nabla \mac{\mu} = t^{n(\mu )}q^{n(\mu ')} \mac{\mu},
\end{equation*}
where $n(\mu) = \sum (i-1)\mu_i$ and $\mu'$ is the conjugate partiton to 
$\mu$.

The following theorem, which we address by \textit{the $q,t$-Catalan theorem}is a culmination of the works of many mathematicians.

\begin{thm}\label{T:qtCat}(Garsia, Haglund, Haiman, \cite{GaHag02},\cite{GaHa96}, \cite{Hai98}, \cite{Hai02}) \label{T:qtCat}
For all $n\geq 1$, the Hilbert series $\Hs{R_n^{\varepsilon}}(q,t)$ of the space of diagonal
alternating coinvariants $R_n^{\varepsilon}$ is equal to each of the following quantities below:
\begin{enumerate}
\item The multiplicity $\scprod{\fro{R_n}{q,t}}{e_n}=\scprod{\nabla e_n}{e_n}$ of the sign character $s_{1^n}=e_n$ in the $\mathfrak{S}_n$-module $R_n$, where $\langle, \rangle$ is the Hall scalar product on the space
of symmetric functions.   
\item The $q,t$-Catalan series
\begin{equation*}
C_n(q,t) = \sum_{P\in \Dy{n}} q^{s_1(P)}t^{s_2(P)},
\end{equation*}
where $s_1$ and $s_2$ are the area and bounce functions on $\Dy{n}$. 
\item The rational function
\begin{equation*}
\sum_{\mu \in \y{n}} \frac{ t^{n(\mu )}q^{n(\mu ')}(1-q)(1-t) \Pi _{\mu }(q,t) B_{\mu
}(q,t) }{ \prod _{\zeta \in \mu} (1 - t^{1+l(\zeta)}q^{-a(\zeta)}) (1 -
t^{-l(\zeta)}q^{1+a(\zeta)})},
\end{equation*}
where $B_{\mu}(q,t) = \sum _{(r,s)\in \mu}t^{r}q^{s}$, $\Pi _{\mu }(q,t)= \prod _{(r,s)\in \mu\setminus (0,0)} (1-t^{r}q^{s})$, and $a(\zeta), l(\zeta)$ are the lengths of the arm and leg (respectively)
of the cell $\zeta\in\mu$.

\item The Euler characteristic of a certain $\TT^2$-equivariant sheaf supported on the zero fiber $Z_n$
of the Hilbert scheme $\Hil{n}$.

\end{enumerate}
\end{thm}

 In \cite{Hai98}, Haiman showed that $\Hil{n}$ can be realized as a blow up of the $n^{th}$ symmetric product $S^n(\aff) :=(\aff)^n/\mathfrak{S}_n$ of the plane $\aff$ along a particular subscheme. Let $\shf{}(1)$ be the ample sheaf on $\Hil{n}$ arising from the $\Proj$ construction of the blow up. Then the suitable sheaf giving the fourth item in the $q,t$-Catalan theorem is $M:=\shf{}(1)\otimes \shf{Z_n}$, where $Z_n$ is the zero fiber in $\Hil{n}$ -the subscheme consisting of all points that are supported at the origin of $\aff$. 

Based on the computer experiments we conjecture that

\begin{conj}\label{T:nestedcharacter} Let $\eta:\Hil{n,n-1} \rightarrow \Hil{n}$ be the projection sending $(I_1,I_2)$ to $I_1$. 
Then for every $m \geq 1$, the Euler characteristic 
\begin{equation}
\chi_{\eta^*(\shf{}(m)\otimes \shf{Z_n})}(q,t)
\end{equation}
of the pull back of the sheaf $\shf{}(m)\otimes \shf{Z_n}$ on $\Hil{n,n-1}$ is a polynomial in $q$ and $t$ with nonnegative integer coefficients. Furthermore, letting $m=1$ we conjecture that 
\begin{equation}
\chi_{\eta^*(\shf{}(1)\otimes \shf{Z_n})}(q,t)=N_n(q,t),
\end{equation}
where the right hand side is given by \ref{E:nestedcatalan} above.
\end{conj}

We shall say more about this conjecture in the section \ref{S:specialcase}.

Now, we would like to sketch how the exposition is organized. 
In section \ref{S:notation} we generated our notation necessary for the rest of the paper. In section \ref{S:cotangent}, we calculate an explicit basis for the cotangent spaces at the torus fixed points of the Hilbert scheme of points in the plane. The notation set in this section shall be used in the preceeding section \ref{S:entersnested} to calculate a basis for the cotangent spaces of the nested Hilbert scheme $\Hil{n,n-1}$ at the torus fixed points. In section \ref{S:zerofiber}, we summarize a few fundamental results concerning the zero fibers $Z_n$ and $Z_{n,n-1}$. We show in particular in that section that $Z_{n,n-1}$ is Cohen-Macaulay. In section \ref{S:characterformulas}, we first summarize the character formulas that Haiman derived, and then we prove the theorem 
\ref{T:nestedABL}. In the next section we look at a special case and present our conjecture \ref{T:nestedcharacter} above. In the final section, we shall talk about a Pierri rule for the Macdonald polynomials which is brought to our attention by Haiman.

\section{\textbf{Notation and the background}}\label{S:notation}

Let us begin with fixing some of our basic notation.
Letters $n,m,i,j,k,l,r,s$ and Greek letters with these letters as subscripts i.e., $\lambda_i, \mu_j$ are reserved for integers. A partition $\lambda$ of a positive integer $n$ is a decreasing sequence of numbers $(\lambda_1,...,\lambda_k)$ such that $\sum \lambda_i = n$. In this case, we shall write $\lambda \in \y{n}$. Often $\mu$ and $\nu$ shall also be used for denoting partitions of some natural numbers.

Without calling it after a country (and hoping not to offend anybody), we shall depict the shape (aka. the diagram) of a partition $\mu$ as in the figure \ref{F:armleg} below.
We shall denote a cell in the diagram of a partition by one of the letters $\zeta$ or $\alpha$, and by abuse of notation we shall write $\zeta \in \mu$ (or $\alpha\in \mu$). For each cell $\zeta$ in the Ferrers diagram of $\mu$, we define the leg $l(\zeta)$, the arm $a(\zeta)$, the co-leg $l'(\zeta)$, and the co-arm $a'(\zeta)$ of $\zeta\in \mu$ to be respectively the number of cells above, to the right of, below, and to the left of $\zeta$. 

We shall use the symbol $\row(\zeta)$ for the set of all cells of $\mu$ in the same row as $\zeta$, other than $\zeta$ itself. Similarly $\col(\zeta)$ shall be used for the set of all cells of $\mu$ in the same column as $\zeta$, other than $\zeta$ itself.

\begin{figure}[htp]
\begin{center}
\scalebox{.5}{\input{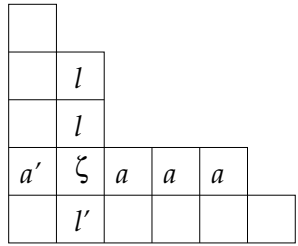}}
\caption{$a(\zeta)=3$, $l(\zeta)=2$, $a'(\zeta)=1$, $l'(\zeta)=1$}
\label{F:armleg}
\end{center}
\end{figure}

The coordinates of a cell $\zeta \in \mu$ are given by the coarm and the colength of $\zeta$, i.e.,
$(a'(\zeta),l'(\zeta))$. A word of caution is ready here. This is not the usual Cartesian coordinates on a point $P\in \N^2$. If the usual Cartesian coordinates of $P$ is $(a,b)$, then with our convention of the coordinates, it is equal to $(b,a)$. Following this convention of coordinates, given a partition $\mu \in \y{n}$, we define the associated set of monomials by $\basis{\mu}= \{x^hy^k: (h,k)\in \mu\}$. For example, $\mathfrak{B}_{(6,5,2,2,1)}$ contains the monomials
\begin{equation*}
\begin{array}{cccccc}
x^4 & & & & & \\
x^3 & x^3 y & & & & \\
x^2 & x^2y & & & & \\
x & xy & xy^2 & xy^3 & xy^4 &\\
1 & y & y^2 & y^3 & y^4 & y^5 
\end{array}
\end{equation*}

Let $\mathbf{x}=\{x_1,...,x_n\}$ and $\mathbf{y}=\{y_1,...,y_n\}$ be two sets of variables. By abuse of notation, we shall often use the symbol $\C[\mathbf{x},\mathbf{y}]$ to denote the polynomial ring $\C[x_1,y_1,...,x_n,y_n]$.

\subsubsection{Hilbert schemes}
A Hilbert scheme on a projective variety $X$ is a scheme that parametrizes all subschemes of $X$ with the prescribed ''Hilbert polynomial.'' In order not to digress too much We shall not explain what the Hilbert polynomial of a variety is. The existence of a Hilbert scheme has been proved by Grothendieck \cite{Grot61}. In the special case of Hilbert schemes of points, there are various elementary proofs of the existence (see \cite{GustavsenEtal}, \cite{Huibregtse06}). Using deformation techniques, Hartshorne \cite{Hartshorne66} showed that the Hilbert schemes are connected. For the special case of Hilbert schemes on surfaces, in \cite{Fog68} Fogarty showed that $\Hil{n}$ is smooth and connected.

It follows from basics of algebraic geometry that a closed point of $\Hil{n}$ can be identified with an ideal $I\subseteq \C[x,y]$ of length $n$. This means that the quotient $\C[x,y]/I$ is an $n$-dimensional vector space over $\C$. Therefore, set theoretically we can write
\begin{equation*}
\Hil{n}=\{I \subseteq \C[x,y]:\ dim_{\C}(\C[x,y]/I)=n\}.
\end{equation*}

Let $\mu\in \y{n}$ be a partition of $n$. Consider the set 
\begin{equation}
U_{\mu} := \{ I\in \Hil{n}:\ \basis{\mu}\ \text{is a vector space basis for}\ \C[x,y]/I\}.
\end{equation}

Note that this set is not empty because the monomial ideal 
\begin{equation*}
I_{\mu} := (x^hy^k:\ (h,k)\notin \basis{\mu})
\end{equation*}
is in $U_{\mu}$.

Notice that a minimal set of generators for the ideal $I_{\mu}$ is determined by the diagram of the partition $\mu$. 

\begin{example}\label{e:ideal3211} Let $\mu=(3,2,1,1)$, then $\basis{\mu}=\{1,x,x^2,x^3,y,xy,y^2 \}$, then the set $\{x^4,x^2y,xy^2,y^3\}$ constitutes a minimal generating set for $I_{\mu}=(x^hy^k:\ (h,k)\notin \basis{\mu})$.
\end{example}

\begin{figure}[htp]
\begin{center}
\scalebox{.5}{\input{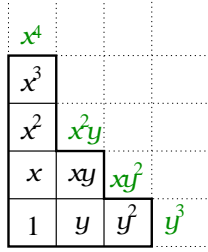}}
\caption{$I_{(3,2,1,1)}$ is generated by $x^4,x^2y,xy^2,y^3$}
\label{F:generators}
\end{center}
\end{figure}

In \cite{Huibregtse02}, Huibregtse has studied the sets $U_{\mu}$ in great detail.
We shall list a few basic properties of them. For the proofs, we refer the interested reader to \cite{Huibregtse02} and \cite{Hai98}. 

First of all, the sets $U_{\mu}$ are open affine subvarieties of $\Hil{n}$, and as $\mu$ runs over all partitions of $n$ they form a covering of the Hilbert scheme,  
\begin{equation*}
\Hil{n}= \bigcup_{\mu \in \y{n}} U_{\mu}.
\end{equation*} 

Now, let $I\in U_{\mu}$ be an arbitrary ideal, and let $x^hy^k$ be an arbitrary monomial from $\C[x,y]$. Then, modulo $I$, this monomial can be 
written as a linear combination (possibly with zero coefficients) 
\begin{equation}\label{E:Fundamentalequations}
x^hy^k = \sum_{(r,s)\in \mu} c^{hk}_{rs}(I) x^ry^s\hspace{.1in} (\text{mod}\ I),
\end{equation}
where $c^{hk}_{rs}(I)\in \C$. It turns out that $c^{hk}_{rs}$ are regular functions on $U_{\mu}$ and the algebra 
\begin{equation*}
\C[c^{hk}_{rs}:\ (h,k)\notin \mu]
\end{equation*}
is equal to the coordinate ring $\shf{\Hil{n}}(U_{\mu})$ of $U_{\mu} \subseteq \Hil{n}$.

\subsubsection{Some relevant schemes}
Due to its construction, a Hilbert scheme acquires a universal family which projects on it. In the case of Hilbert scheme of points in the plane (see proposition 2.9 in \cite{Hai98}), the universal family is the reduced, closed subvariety $F_n \subseteq \Hil{n}\times \aff$, which satisfying the following universal property;

\textit{Let $T$ be a scheme parametrizing a flat family $S$ of length $n$ subschemes of $\aff$. Then there exists a unique morphism $T\rightarrow \Hil{n}$ such that}
\begin{equation*} 
S = T \times_{F_n} \Hil{n}.
\end{equation*}

Following definitions, it can easily be shown that 
at least set theoretically $F_n$ is equal to 
\begin{equation}\label{E:universal}
F_n=\{(I,P)\in \Hil{n}\times \aff: P \in V(I)\},
\end{equation}
where $V(I)$ denotes the vanishing locus of the ideal $I\in \Hil{n}$ in $\aff$. Note here that since $I$ corresponds to a zero dimensional, length $n$ subscheme of $\aff$ counting with multiplicity $V(I)$ has exactly $n$ points.

Note also that the projection $\pi:F_n\longrightarrow \Hil{n}$ is a finite, 
flat morphism of rank $n$. Hence, pushdown $B_n=\pi_*(\shf{F_n})$ onto $\Hil{n}$ of the structure sheaf -sometimes called as ``a tautological sheaf'' on $\Hil{n}$- is a vector bundle of rank $n$. In fact more is true; $B_n$ is a bundle of $\shf{\Hil{n}}$-algebras with fibers $\C[x,y]/I$. 

Along with the universal family, $n$-fold symmetric power $S^n (\aff)$ plays a fundamental role in the study of the Hilbert scheme of points. This is the scheme of unordered $n$-tuples of points from $\aff$. We shall often denote a point of $S^n(\aff)$ by  $\llbracket p_1,...,p_{n}\rrbracket$, where $p_i\in \aff$. The symmetric power $S^n(\aff)$ can also be identified as the space of cycles of length $n$ in the plane
\begin{equation*}
S^n(\aff) = \{ \sum_{P\in \aff} a_P [P]:\ a_P\in \N,\ \sum a_P = n\},
\end{equation*}
where the notation $[P]$ stands for the divisor defined by the point $P\in aff$.  

The Chow morphism $\sigma:\Hil{n}\longrightarrow S^n(\aff)$ is 
defined by sending an ideal $I$ to its corresponding zero cycle $\sigma(I)=\sum_{P\in V(I)}m_P [P ]$, 
where $m_P$ is the multiplicity of $P$ in the scheme $\Spec (\C[x,y]/I)$. It is easy to see that the 
symmetric power $S^n(\aff)$ is the affine scheme $\Spec \C[\mathbf{x},\mathbf{y}]^{\mathfrak{S}_n}$ over the ring of invariants 
$\C[\mathbf{x},\mathbf{y}]^{\mathfrak{S}_n}:=\C[x_1,y_1,\dots,x_n,y_n]^{\mathfrak{S}_n}$ of the diagonal action of the symmetric group $\mathfrak{S}_n$. 
By a theorem of Weyl, the maximal ideal corresponding to the ``origin'' $n\cdot [0]$ is generated by the 
\textsl{polarized power sums} $\{p_{rs}:\ r+s>0\}$ where 
\begin{equation}\label{E:powersums}
p_{rs}=\sum_{i} x_i^ry_i^s \in \C[\mathbf{x},\mathbf{y}]^{\mathfrak{S}_n}.
\end{equation}
We note that this maximal ideal is nothing but $I_+$ (the ideal generated by the bi-symmetric polynomials without a constant term) mentioned in the introduction and that the elements of 
$I_+$ are global regular functions vanishing at the origin $n\cdot [0]$ of $S^n(\aff)$.

So called the \textit{zero fiber}, $Z_n\subseteq \Hil{n}$ is the preimage  $\sigma^{-1}( n\cdot [0])$ of the origin $n\cdot [0] \in S^n(\aff)$. Hence, pull back of any bivariate symmetric polynomial from $I_+ \subseteq \C[\mathbf{x}, \mathbf{y}]$ to $\Hil{n}$ under the Chow morphism, 
vanishes on $Z_n$, in particular $\sigma^* p_{r,s}$ has to vanish on $Z_n$.

\section{\textbf{Cotangent space vs. tangent space}}\label{S:cotangent}

In this section, we will review descriptions of the cotangent and the tangent spaces at the points of $\Hil{n}$ which are defined by monomial ideals. We first state a theorem of Haiman describing a basis for the cotangent space at a monomial ideal. Then we shall write a particular basis for the tangent space at a monomial ideal which we shall use while describing the tangent spaces of the nested Hilbert scheme $\Hil{n,n-1}$. The reason for why we do choose a particular basis is explained in the remark at the end of the section.

We start by the cotangent spaces. Recall that the coordinate ring of the open set $U_{\mu}$ is generated (as an algebra) by the functions $c^{rs}_{hk}$ where $(r,s)\notin \mu$. In \cite{Hai98}, Haiman calculates the following basis for the cotangent space at the monomial ideal $I_{\mu}$.

\begin{prop}\cite{Hai98}
The cotangent space $\mathfrak{m}/\mathfrak{m}^2$ of $\Hil{n}$ at the monomial ideal $I_{\mu}$ is spanned by the images of 
the following $2n$ coordinate functions from the maximal ideal $\mathfrak{m}$ at $I_{\mu}\in \Hil{n}$
\begin{equation}
u_{hk} = c^{h,g+1}_{f,k},\ d_{hk} = c^{f+1, k}_{h,g},
\end{equation}
where $(f,k)\in \mu$ is the top cell in the column of $(h,k)\in \mu$, and $(h,g)\in \mu$ is the right-most cell in the row of $(h,k)$. 
\end{prop}

\begin{rem}\label{R:torusaction} The torus $\TT^2$ acts algebraically on $\C[x,y]$ by $(t,q)\cdot x = tx$ and $(t,q)\cdot y = qy$. This action passes onto $\Hil{n}$. More precisely, on $\Hil{n}$ it is defined by 
\begin{equation*}
(t,q)\cdot I = \{p(t^{-1}x,q^{-1}y):\ p(x,y)\in I\},\ \text{for}\ I\in \Hil{n}.
\end{equation*}
Notice that an ideal $I\in \Hil{n}$ is fixed under this action if and only if $I$ is doubly homogeneous and this is equivalent to say that $I$ is spanned by monomials. Therefore, the torus fixed points of the Hilbert scheme $\Hil{n}$ consists of the ideals of the form
\begin{equation*}
I_{\mu} = (x^hy^k:\ (h,k)\notin \basis{\mu}),\ \text{for}\ \mu\in \y{n}.
\end{equation*}
At a torus fixed point the equations \ref{E:Fundamentalequations} has to stay unchanged, hence we must have 
\begin{equation*}
(t,q)\cdot c^{rs}_{hk}=t^{h-r}q^{k-s}c^{rs}_{hk}.
\end{equation*}
Therefore, for a cell $\zeta=(h,k)$ in the diagram of the partition $\mu$ the torus character of the vector $d_{hk}$ is $t^{1+l(\zeta)}q^{-a(\zeta)}$, while the character of $u_{hk}$ is $t^{-l(\zeta)}q^{1+a(\zeta)}$.
\end{rem}

Let $X$ be an arbitrary Hilbert scheme over an affine scheme $\Spec A$. Grothendieck \cite{Grot61} showed that the tangent space at a point $z\in X$ with the corresponding ideal $I_z\subseteq A$, 
is canonically isomorphic to
\begin{equation}
\Hom_{A}(I_z, A/I_z),
\end{equation}
considered as a vector space over the residue field $\shf{X,z}/\mathfrak{m}_z$ of the point $z$. In our case, 
$A=\C[x,y]$, and $P$ is the constant polynomial $n$. 
We denote this tangent space by
\begin{equation*}
T_{\Hil{n},I} := \Hom_{\C[x,y]}(I, \C[x,y]/I).
\end{equation*}

Now, let $I_{\mu}$ be a monomial ideal and let $\gamma\in T_{\Hil{n},I_{\mu}}$ be a homomorphism. Recall that the 
generators of $I_{\mu}$ are monomials that sit on the outer corners of the partition $\mu$. We are going to 
call these monomials as \textit{canonical generators} and denote the set of all canonical generators by $\mathfrak{C}_{\mu}$. For example, for the ideal in the example \ref{e:ideal3211} the set $\mathfrak{C}_{(3,2,2,1)}$ of canonical generators is $\{x^4,x^2y,xy^2,y^3\}$ (see also figure \ref{F:generators}).

Recall that $\basis{\mu}$ is a basis for $\C[x,y]/I_{\mu}$. Therefore, $\gamma$ is uniquely determined by the images of the canonical generators in the span of $\basis{\mu} \mod I$
 Notice however that an assignment $\gamma$ cannot be arbitrary since $\gamma$ has to be an $\C[x,y]$-module homomorphism. 

Indeed, if $\gamma(x^ry^s)=g \mod I_{\mu}$ for a canonical generator $x^ry^s\in \mathfrak{C}_{\mu}$, then we have to have 
\begin{equation}\label{E:hom}
\gamma(x^{r+i}y^{s+j})=x^iy^j\cdot g \mod I_{\mu}
\end{equation}
for any $x^iy^j\in \C[x,y]$.

It will soon be apparent to the reader that what is happening in (\ref{E:hom}) is equivalent to the ``sliding rules'' described by Haiman in \cite{Hai98} of the arrows associated with the generators of the cotangent space.  In the rest of this section, we are going to write a basis for the tangent space to $\Hil{n}$ at the monomial ideal $I_{\mu}$. It will consist of $2n$ homomorphisms (two for each square of the partition $\mu$ of $n$) The first $n$ of these homomorphism will correspond to ``downward arrows'' of \cite{Hai98}, and the latter $n$ homomorphisms will correspond to ``upward arrows.''

We shall start with defining \textit{downward homomorphisms} $\{d^*_{hk}\}$ first. Let $(h,k)\in \mu$ be a cell and let $(f,k)\in \mu$ be the top cell in the column of $(h,k)$. Similarly, let $(h,g)\in \mu$ be the right-most cell in the row of $(h,k)$. 

To ease the understanding, we shall construct downward homomorphism $d^*_{hk}: I_{\mu} \rightarrow \C[x,y]/I_{\mu}$ in several steps. First step in its construction is rather easy to state. We set  
\begin{equation}\label{E:downarrow1}
d_{hk}^*(x^{f+1}y^k) = x^hy^g \mod I_{\mu}.
\end{equation}
It is helpful to visualize (\ref{E:downarrow1}) as if it is an arrow starting at $(f+1,k)$ and ending at $(h,g)$ as in the figure \ref{fig:darrow}. Let us denote this downward arrow by $\mathfrak{d}_{hk}$.

\begin{figure}[h]
\begin{center}
\scalebox{.5}{\input{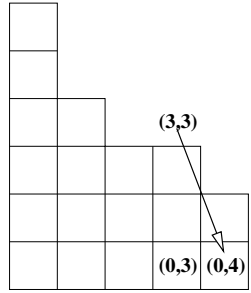}}
\caption{The downward arrow $\mathfrak{d}_{03}$}
\label{F:darrow}
\end{center}
\end{figure}

If $x^{f+1}y^k$ is not a canonical generator, then we would like to ``slide''
$\mathfrak{d}_{hk}$ without changing the relative positions of the head and the tail, as follows. Suppose that for some $l>0$, the monomial $x^{f+1}y^{k-l}$ is a canonical generator in the row of $x^{f+1}y^k$. Then, we define  
$d_{hk}^*(x^{f+1}y^{k-l})=x^hy^{g-l} \mod I_{\mu}$. For the arrow in the figure \ref{fig:darrow}, $x^3y^2$ is a canonical generator for $I_{(5,5,4,2,1,1)}$ and we define $d_{03}^*(x^3y^2)=y^3$.

For the rest of the construction, without loss of generality we shall assume that $x^{f+1}y^k$ is canonical generator.

Now, suppose that $x^ay^b\in \canonical{\mu}$ is another canonical generator with $a>{f+1}$. Note that this implies that $b < f$. Suppose also that without changing the relative positions of the head and the tail we can slide the arrow $\mathfrak{d}_{hk}$ and bring its tail to $x^ay^b$ in such a way that while moving the head always stays inside the diagram of $\mu$ and the tail always stays outside. In that case we shall define  
\begin{equation}
d_{hk}^*(x^ay^b) = x^{a-f-1+h}y^{b+g-k} \mod I_{\mu}.
\end{equation}
If this sliding procedure cannot be done without head of the arrow staying inside the diagram of $\mu$, then we set 
\begin{equation}
d_{hk}^*(x^ay^b) = 0 \mod I_{\mu}.
\end{equation}
Finally, for any canonical generator $x^ay^b \in \canonical{\mu}$ with $b> f+1$, we want
\begin{equation}
d_{hk} ^*(x^ay^b) = 0 \mod I_{\mu}.
\end{equation}
The reason for why we want the last equality is that the arrow $\mathfrak{d}_{hk}$ cannot be slided in the direction that the exponents of $y$ increases. Therefore, in order for not to make homomorphism $d_{hk}^*$ more complicated than it should be, we send those canonical generators to zero. This finishes the desription of the map $d_{hk}^*$.

\begin{example} Let $\mu=(5,5,4,2,1,1)$ be the partition whose diagram is as in the figure \ref{fig:darrow}. Then the set of canonical generators are
\begin{equation*}
\canonical{(5,5,4,2,1)}= \{x^6,x^4y,x^3y^2,x^2y^4,y^5 \}.
\end{equation*}
Now, according to construction above, we must have 
$d_{03}^*(x^6)=0,\ d_{03}^*(x^4y)=xy^2,\ d_{03}^*(x^3y^2)=y^3,\ d_{03}^*(x^2y^4)=0$, and $d_{03}^*(y^5)= 0$ modulo $I_{\mu}$. 
\end{example}

Similar to $d_{hk}^*$,  we shall define $u^*_{hk}: I_{\mu} \rightarrow \C[x,y]/I_{\mu}$ in steps. Once again, let $(h,k)$ be a cell in the diagram of the partition $\mu$. Let $(h,g)$ be the right-most cell in the row of $(h,k)$ in the diagram of $\mu$, and let $(f,k)$ be the top cell in the column of $(h,k)$. First we define
\begin{equation}
u_{hk}^*(x^hy^{g+1}) = x^fy^k \mod I_{\mu}.
\end{equation}
We would like to visualize this assignment as an \textit{upward arrow} as in figure \ref{fig:upward}. Let us denote it $\mathfrak{u}_{hk}$.

\begin{figure}[h]
\begin{center}
\scalebox{.5}{\input{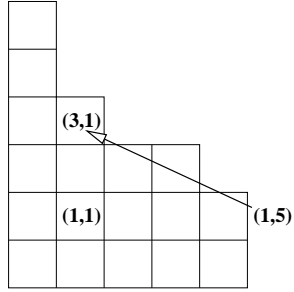}}
\caption{The upward arrow $\mathfrak{u}_{11}$}
\label{fig:upward}
\end{center}
\end{figure}

If $x^hy^{g+1}$ is not a is not a canonical generator, then we would like to slide $\mathfrak{u}_{hk}$ downwards without changing the relative positions of the head and the tail. Suppose that for some $l>0$, the monomial $x^{h-l} y^{g+1}$ is a canonical generator in the column of $x^h y^{g+1}$. Then, we define $u_{hk}^*(x^{h-l} y^{g+1})=x^{f-l} y^k \mod I_{\mu}$. For the arrow in the figure \ref{fig:upward}, $y^5$ is a canonical generator for $I_{(5,5,4,2,1,1)}$ and we define $u_{11}^*(y^5)=x^2y$. 

For the rest of the construction, without loss of generality we shall assume that $x^hy^{g+1}$ is canonical generator.

Suppose that $x^ay^b\in \canonical{\mu}$ is another canonical generator with $b>{g+1}$. Note that this implies that $a < h$. Suppose also that without changing the relative positions of the head and the tail we can slide the arrow $\mathfrak{u}_{hk}$ and bring its tail to $x^ay^b$ in such a way that while moving the head always stays inside the diagram of $\mu$ and the tail always stays outside. In that case we shall define  
\begin{equation}
u_{hk}^*(x^ay^b) = x^{f-h+a}y^{k+b-g+1} \mod I_{\mu}.
\end{equation}
If this sliding procedure cannot be done without head of the arrow staying inside the diagram of $\mu$, then we set 
\begin{equation}
u_{hk}^*(x^ay^b) = 0 \mod I_{\mu}.
\end{equation}
Finally, for any canonical generator $x^ay^b \in \canonical{\mu}$ with $a> h$, we want
\begin{equation}
u_{hk} ^*(x^ay^b) = 0 \mod I_{\mu}.
\end{equation}
This finishes the description of the map $u_{hk}^*$. 

\begin{example} Let $\mu=(5,5,4,2,1,1)$ be the partition whose diagram is as in the figure \ref{fig:darrow}. Then the set of canonical generators are
\begin{equation*}
\canonical{(5,5,4,2,1)}= \{x^6,x^4y,x^3y^2,x^2y^4,y^5 \}.
\end{equation*}
Now, according to construction above, we must have 
$u_{11}^*(x^6)=0,\ u_{11}^*(x^4y)=0,\ u_{11}^*(x^3y^2)=0,\ u_{11}^*(x^2y^4)=0$, and $u_{11}^*(y^5)= x^2y$ modulo $I_{\mu}$. 
\end{example}

\begin{prop}
Tangent space $T_{\Hil{n},I_{\mu}} = \Hom_{\C[x,y]}(I_{\mu}, \C[x,y]/I_{\mu})$ of $\Hil{n}$ at a monomial ideal 
$I_{\mu}$ is spanned as a vector space by the set $\{d^*_{hk},u^*_{hk}:\ (h,k)\in \mu\}$ of $2n$ homomorphisms.
\end{prop}

\begin{proof}
Linear independence is clear from the fact that the images of the monomials $x^{f+1}y^k$ and $x^hy^{g+1}$ 
are standard monomials which form a basis modulo $I_{\mu}$. Since dimension of a tangent space is the same 
as the dimension of the cotangent space, we are done.
\end{proof}

As mentioned in the remark \ref{R:torusaction}, one can calculate the eigenvalues of the torus action on the cotangent space. On the tangent space we can do it as follows. Let $\zeta=(h,k)$ be a cell in the diagram of $\mu$. Since $d_{hk}^*(x^{f+1}y^k) = x^hy^g \mod I_{\mu}$, we must have 
\begin{equation}
(t,q)\cdot d_{hk}^* = t^{h-(f+1)}q^{g-k} d_{hk}^* = t^{-1-l(\zeta)}q^{a(\zeta)}d_{hk}^*.
\end{equation}
Similarly, since $u_{hk}^*(x^hy^{g+1}) = x^fy^k \mod I_{\mu}$, we must have 
\begin{equation}
(t,q)\cdot u_{hk}^*= t^{f-h}q^{k-(g+1)} u_{hk}^* = t^{l(\zeta)}q^{-1-a(\zeta)}u_{hk}^*. 
\end{equation}

Notice that these eigenvalues are exactly the reciprocals of the eigenvalues that Haiman has calculated. At this point one might wonder if the basis 
$\{d_{hk}, u_{hk}\}_{\mu}$ of the cotangent space is the dual basis to the basis $\{d^*_{hk},u^*_{hk} \}_{\mu}$ of the tangent space. We shall not worry about this problem. Instead we define a dual basis for the cotangent space.

\begin{defn}
Let $T^*_{\Hil{n},I_{\mu}}$ be the cotangent space of $\Hil{n}$ at the torus fixed point $I_{\mu}$. Recall that this is just the dual vector space of the tangent space $T_{\Hil{n},I_{\mu}}$. We shall denote the elements of the dual basis to $\{d_{hk}^*, u_{hk}^* \}_{\mu}$ by $\{D_{hk},U_{hk}\}_{(h,k)\in \mu}$, where
\begin{equation}
D_{hk}(d_{h'k'}) = \delta_{hk,h'k'},\ \text{and}\ U_{hk}(u_{h'k'})=
\delta_{hk,h'k'},
\end{equation}
and $\delta_{hk,h'k'}$ is the Kronecker delta. 
\end{defn}

\begin{rem}\label{R:eigenvalues} Recall the following basic but easily-forgettable fact that an invertible linear transformation on a vector space acts on the dual by its inverse,
i.e., if $T:V\rightarrow V$ is a transformation then $T$ acts on $f\in V^*$ by $(Tf)v = f(T^{-1}(v))$. 
Therefore, for a cell $\zeta=(h,k)\in \mu$, we see that the torus eigenvalue of $D_{hk}$ is $t^{1+l(\zeta)}q^{a(\zeta)}$, while the torus eigenvalue of $U_{hk}$ is $t^{l(\zeta)}q^{1+a(\zeta)}$.   
\end{rem}

\section{\textbf{Enters the smooth nested Hilbert scheme of points}}\label{S:entersnested}

In this section we shall concentrate on a reduced closed subscheme $\Hil{n,n-1} \subseteq \Hil{n} \times \Hil{n-1}$ defined by incidence relations
\begin{equation*}
\Hil{n,n-1} := \{ (I_1,I_2)\in \Hil{n}\times \Hil{n,n-1}:\ I_1\subseteq I_2\}.
\end{equation*}
This scheme is called as the smooth nested Hilbert scheme of points in the plane. The torus action on $\Hil{n}$ naturally extends to $\Hil{n,n-1}$. At the end of the section we shall describe a basis for the tangent spaces at the torus fixed points for $\Hil{n,n-1}$. We should mention here that Cheah \cite{Cheah98} computes the dimension of a tangent space at a torus fixed point to be $2n$. However, she does not write a basis in the explicit fashion as we are going to write here. We shall start with some basic facts about the nested Hilbert schemes.

Given a non increasing sequence $n_1\geq \cdots \geq n_r$ of positive integers one can consider the nested Hilbert scheme of points $\Hil{n,n_1,...,n_r}$ defined by 
\begin{equation*}
\Hil{n_1,...,n_r} := \{(I_1,...,I_r)\in \Hil{n_1}\times \cdots \times \Hil{n_r}:\ I_1\subseteq \cdots \subseteq I_r\}.
\end{equation*}

Observe that the dimension of the generic locus of $\Hil{n_1,...,n_r}$ (where the smallest ideal $I_1$ corresponds to $n_1$ distinct points) is $2n_1$, because on the generic locus there is a finite to one mapping onto $\mathbb{A}^{n_1}$. Therefore, $\Hil{n_1,n_2,...,n_r}$ is of dimension $2n_1$. It is an interesting fact due to Cheah \cite{Cheah98} that nested Hilbert schemes of points in the plane, except $\Hil{n,n-1}$ and few other small cases, are singular. From the combinatorial point of view, it might be interesting to determine what kind of singularities that $\Hil{n_1,n_2,...,n_r}$ might have. On that note, the simplest example of a nested Hilbert scheme of points (other than $\Hil{n,n-1}$) whose singularities are easy to understand is $\Hil{n,1}$.

\begin{prop}
The nested Hilbert scheme $\Hil{n,1}$ is Cohen Macaulay of dimension $2n$.
\end{prop}

\begin{proof}
Let $(I_1,I_2)$ be a point of $\Hil{n,1}$. Then $V(I_2)\subseteq V(I_1)$ and $I_2$ is of colength 1. In other words, 
$I_2$ is a maximal ideal corresponding to a point $p\in \aff$. Then, we can define the morphism $\Hil{n,1} \rightarrow F_n$ 
to the universal family $F_n$ over $\Hil{n}$ by sending $(I_1,I_2)$ to $(I_1,p)$. Clearly, this is an isomorphism. 
Since $F_n$ is finite flat over the smooth scheme $\Hil{n}$, $F_n$ and hence $\Hil{n,1}$ are Cohen Macaulay.
\end{proof}

Since in the rest of the paper we are not going to talk about any nested Hilbert scheme of points but $\Hil{n,n-1}$, we are going to call $\Hil{n,n-1}$ as \textit{the nested Hilbert scheme} (by dropping the adjective ``smooth'').  The following theorem about the nested Hilbert scheme has been attributed to Tikhomirov (unpublished).

\begin{thm}
The nested Hilbert scheme $\Hil{n,n-1}$ is nonsingular and irreducible of dimension $2n$. 
\end{thm}

Let $(I_1,I_2)$ be a point of $H_{n,n-1}$, and let $\tau=(t,q) \in \TT^2$. Then $(\tau \cdot I_1, \tau \cdot I_2)\in \Hil{n,n-1}$, where the action of $\tau$ on an ideal $I$ is as in the remark \ref{R:torusaction}. It is easy to see that a torus fixed point has to be a pair of monomial ideals $(I_1,I_2)$ of $\Hil{n,n-1}$ such that $I_1=I_{\mu}\subsetneq I_2=I_{\mu'}$, where $\mu \in \y{m}$ and $\mu' \in \y{m-1}$. 

Recall that the $n'th$ symmetric product $S^n(\aff)$ as a set of cycles is $S^n(\aff) = \{ \sum_{P\in \aff} a_P [P]:\ a_P\in \N,\ \sum a_P = n\}$, and the Chow morphism(s) $\sigma_i:\Hil{i} \rightarrow S^i(\aff)$, $i\in\{n,n-1\}$ are defined by 
\begin{equation*}
\sigma_i(I) = \sum_{P\in V(I)} n_P [P],\ I\in \Hil{i}
\end{equation*}
where $n_P$ is the multiplicity of the point $P$ in $V(I)$.

Now, if $(I_1,I_2)\in \Hil{n,n-1}$, then $\sigma_{n-1}(I_2)$ is an $n-1$ element subset of $\sigma_n(I_1)$. In other words, if 
\begin{equation*}
\sigma_{n-1}(I_2)= \sum_{i=1}^{n-1} [p_i] \in S^{n-1}(\aff),\ p_i\in V(I_2)
\end{equation*}
then 
\begin{equation*}
\sigma_n(I_1)= (\sum_{i=1}^{n-1} [p_i] )+ [p_n] \in S^n(\aff),
\end{equation*}
where $p_n$ is the \textit{extra point} of $V(I_1)$.

Note that $p_n\in \aff$ might lie in $V(I_2)$ as a point from $\aff$. In any case, it allows us to define the morphism $\rho: \Hil{n,n-1} \rightarrow F_n$ from nested Hilbert scheme onto the universal family $F_n$ on $\Hil{n}$. Recall that $F_n$, as a set
consists of pairs $(I_1,p)\in \Hil{n}\times \aff$ with $p\in I_1$. Hence, the morphism is defined by sending $(I_1,I_2)$ to $(I_1,p_n)$. It is clear that on the generic locus this is an isomorphism.

Let $(I_1,I_2)\in \Hil{n,n-1}$, and let $\mu \in \y{n}$ be a partition of $n$ such that $I_1\in U_{\mu}$. The containment $I_1\subseteq I_2$ induces a vector space injection in the opposite direction
\begin{equation}
\C[x,y]/I_2 \hookrightarrow \C[x,y]/I_1.
\end{equation}
Therefore, if $\basis{\mu}$ is the basis for $\C[x,y]/I_1$, then $\basis{\mu'}$ is a canonical 
basis for $\C[x,y]/I_2$ where $\mu'$ is a partition obtained from $\mu$ by taking a corner cell $\zeta=(h,k)$ 
off from $\mu$ (we shall write $\mu'=\mu \setminus \zeta$). Hence, if $I_1$ belongs to the open $U_{\mu}$, then $I_2$ belongs to $U_{\mu \setminus \zeta}$.

Clearly, in order for studying the Zariski (co)tangent space at $(I_1,I_2)\in \Hil{n,n-1}$, one might use the Pl\"{u}cker 
relations among the generators of the coordinate rings $\shf{\Hil{n}}(U_{\mu})$ and $\shf{\Hil{n-1}}(U_{\mu\setminus \zeta })$. 
Alternatively, as Cheah does in \cite{CheahThesis}, one might exploit the Grothendieck's description of the 
tangent space for the Hilbert schemes. Thus, it turns out that the tangent space of $\Hil{n,n-1}$ at a point $(I_1,I_2)$ 
is canonically isomorphic to the kernel of a homomorphism between two vector spaces described below. 

Let $\phi:\Hom_{\C[x,y]}(I_1, \C[x,y]/I_1) \rightarrow \Hom_{\C[x,y]}(I_1, \C[x,y]/I_2)$ be 
the homomorphism defined by the composition 
\begin{equation}
I_1 \rightarrow \C[x,y]/I_1 \twoheadrightarrow \C[x,y]/I_2,
\end{equation}
and let $\psi: \Hom_{\C[x,y]}(I_2, \C[x,y]/I_2) \rightarrow \Hom_{\C[x,y]}(I_1, \C[x,y]/I_2)$
be the homomorphism defined by the composition
\begin{equation}
I_1 \hookrightarrow I_2 \rightarrow \C[x,y]/I_2.
\end{equation}
Then the kernel of the combined homomorphism
\begin{equation*}
\phi-\psi: \Hom_{\C[x,y]}(I_1, \C[x,y]/I_1) \oplus \Hom_{\C[x,y]}(I_2, \C[x,y]/I_2) \rightarrow \Hom_{\C[x,y]}(I_1, \C[x,y]/I_2),
\end{equation*}
is canonically isomorphic to the tangent space of $\Hil{n,n-1}$ at $(I_1,I_2)$. We shall now describe an explicit basis for the tangent spaces at the torus fixed points.

To this end, let $(I_{\mu},I_{\mu \setminus \zeta})$ be a torus fixed point, also let $\{u_{hk}^*, d_{hk}^*:\ (h,k)\in \mu\}$ and $\{u_{h'k'}^*, d_{h'k'}^*:\ (h',k')\in \mu\setminus \zeta\}$
be the bases for the tangent spaces at $I_{\mu}$ and $I_{\mu \setminus \zeta}$, respectively.

Recall that $\row(\zeta)$ (and $\col(\zeta)$) stands for the set of all cells of $\mu$ that are in the same row (resp. column) as $\zeta$, other than $\zeta$ itself.

We first form the following 8 disjoint subsets of $T_{\Hil{n},I_{\mu}} \oplus T_{\Hil{n-1},I_{\mu\setminus \zeta}} $. 

\begin{enumerate}
\item $A_1 := \{(d_{hk}^*,0):\ (h,k)\in \row(\zeta)\}$.
\item $A_2 := \{(u_{hk}^*,0):\ (h,k)\in \col(\zeta)\}$.
\item $A_3 := \{(u^*_{\zeta},0)\}$.
\item $A_4 := \{(0,d^*_{\zeta})\}$.
\item $A_5 := \{(0,d_{hk}^*):\ (h,k)\in 
\row(\zeta)\}$.
\item $A_6 := \{(0,u_{hk}^*):\ (h,k)\in \col(\zeta)\}$.
\item $A_7 := \{(d_{hk}^*,d_{hk}^*):\ (h,k)\in \mu \setminus 
\row(\zeta)\cup \col(\zeta)\}$.
\item $A_8 := \{(u_{hk}^*,u_{hk}^*):\ (h,k)\in \mu \setminus 
\row(\zeta)\cup \col(\zeta)\}$.
\end{enumerate}

Here, the third and the fourth sets, each of which has a single element account for the arrows at the corner cell $\zeta$. Thereby we justify the use of the symbols $u^*_{\zeta}$ and $d^*_{\zeta}$.

\begin{prop}
The set $A_1\cup \cdots \cup A_8$ of pairs of homomorphisms form a basis for the tangent space of 
$\Hil{n,n-1}$ at the torus fixed point $(I_{\mu},I_{\mu \setminus \zeta})$. 
\end{prop}

\begin{proof}
Linear independence is clear from the fact that at least one of the coordinates of any element from 
$A_1\cup \cdots \cup A_8$ is a basis vector in one of the tangent spaces $T_{\Hil{n},I_{\mu}}$ or 
$T_{\Hil{n-1},I_{\mu \setminus \zeta }}$.  Given that $\Hil{n,n-1}$ is smooth and $2n$ dimensional, 
tangent space $T_{\Hil{n,n-1},(I_1,I_2)}$ is $2n$ dimensional for every point $(I_1,I_2)$ of the nested 
Hilbert scheme $\Hil{n,n-1}$, and therefore, we are done. 
\end{proof}

\section{\textbf{Zero fibers of $\Hil{n}$ and $\Hil{n,n-1}$}\label{S:zerofiber}}

In this section we shall look closely at the zero fibers  $\sigma_n^{-1}(n\cdot [0])\subseteq \Hil{n}$ and $\sigma_{n,n-1}^{-1}((n-1) \cdot [0] , 0) \subseteq \Hil{n,n-1}$, where $\sigma_n$ is the 
Chow morphism for $\Hil{n}$, and $\sigma_{n,n-1}$ is defined as follows. Let $(I_1,I_2)$ be a point of $\Hil{n,n-1}$. 
Then $\C[x,y]$-module $I_2/I_1$ is of length 1, hence, it is isomorphic to $\C[x,y]/\mathfrak{m}$ for some maximal 
ideal $\mathfrak{m}=(x-x_n,y-y_n)\subseteq \C[x,y]$. Then, the \textsl{distinguished point} $p_n\in V(I_1)$ of the 
pair $(I_1,I_2)$ has coordinates $(x_n,y_n)$. We define $\sigma_{n,n-1}:\Hil{n,n-1} \rightarrow S^{n-1}(\aff) \times \aff$ by 
\begin{equation*}
\sigma_{n,n-1}(I_1,I_2) = (\sigma_{n-1}(I_2),(x_n,y_n)).
\end{equation*}
After reviewing some known results, we shall prove that the zero fiber $Z_{n,n-1}$ of $\Hil{n,n-1}$ is Cohen-Macaulay of dimension $n-1$.

One of the earliest results on the zero fiber of $\Hil{n}$ is due to 
Brian{\c{c}}on.
\begin{thm}\cite{Bri77}
The zero fiber $Z_n := \sigma_n(n\cdot [0])\subseteq \Hil{n}$ is an irreducible subvariety of dimension $n-1$. 
\end{thm}

Let $\pi: F_n \rightarrow \Hil{n}$ be the (first) projection from the universal family $F_n$ onto $\Hil{n}$. As a set, the scheme theoretic fiber
\begin{equation*}
\pi^{-1}(Z_n)= \{(I, n\cdot[0])\in F_n:\ I\in Z_n\} \subseteq F_n
\end{equation*}
maps bijectively onto $Z_n$. Since we do not know if the preimage $\pi^{-1}(Z_n) $ is reduced, we can not conclude that these two subschemes are isomorphic to each other. In fact, the isomorphism is true if we consider the reduced scheme structure on $\pi^{-1}(Z_n)$. We shall use the superscript ``red'' on $\pi^{-1}(Z_n)$ to indicate that it is considered as the reduced scheme obtained from $\pi^{-1}(Z_n)$. 

\begin{thm}\cite{Hai98}\label{T:cmHilbert}
The projection $\pi:F_n \rightarrow \Hil{n}$ induces an isomorphism between $\pi^{-1}(Z_n)^{red}$ and $Z_n$. Furthermore, $\pi^{-1}(Z_n)^{red}$ is a local complete intersection, hence, so is $Z_n\subseteq \Hil{n}$. Therefore, $Z_n$ is Cohen Macaulay. 
\end{thm}

In the course of the proof, Haiman shows that the zero fiber $\pi^{-1}(Z_n)^{red}$ is a locally complete intersection defined by the ideal sheaf 
$\mathcal{I}$ given locally on the open subset $\pi^{-1}(U_{\mu})\subseteq U_{\mu}\times \aff \subseteq F_n$, by  
\begin{equation*}
\mathcal{I}(\pi^{-1}(U_{\mu}))= (x,y,p_{hk}:\ (h,k)\in \mu \setminus (0,0)).
\end{equation*}
Here, the function $p_{rs}$ is the pull back of the (polarized) power sum
symmetric function 
\begin{equation*}
x_1^ry_1^s +\cdots x_n^ry_n^s \in \C[x_1,y_1,...,x_n,y_n]^{\mathfrak{S}_n}
\end{equation*}
via the composition of the morphisms $F_n \xrightarrow{\pi} \Hil{n} \xrightarrow{\sigma_n} S^n(\aff)=\Spec \C[\mathbf{x},\mathbf{y}]^{\mathfrak{S}_n}$.

In \cite{Cheah98}, using Bia{\l}ynicki-Birula decomposition Cheah finds a cellular decomposition for the nested Hilbert scheme $\Hil{n,n-1}$ of points. In particular she calculates how many (topological) cells of dimension $i$ exists in the zero fiber $Z_{n,n-1} = \sigma_{n,n-1}^{-1}(n\cdot [0], 0)\subseteq \Hil{n,n-1}$.

\begin{thm} (\cite{Cheah98}, theorem 3.3.5) Let $b_{k,k-1}(i)$ denote the number of $i$-dimensional (topological) cells of $Z_{n,n-1}$. Then
\begin{equation}\label{E:dimension}
\sum_{k=1}^{\infty} \sum_i b_{k,k-1}(i) v^it^k = \frac{t}{1-tv}\prod_{k=1}^{\infty}\frac{1}{1-t^kv^{k-1}}.
\end{equation}
\end{thm}

Using this result, one can easily compute the dimension of $Z_{n,n-1}$. 

\begin{cor}\label{C:dimnested}
The zero fiber $Z_{n,n-1}$ is of dimension $n-1$.  
\end{cor}

\begin{proof}
To see this, we simply expand the product in the equation~(\ref{E:dimension}) into a summation.
\begin{eqnarray*}
\frac{t}{1-tv}\prod_{k=1}^{\infty}\frac{1}{1-t^kv^{k-1}} &=& \frac{t}{1-tv}\cdot \frac{1}{1-t}\cdot \frac{1}{1-t^2v}\cdot \frac{1}{1-t^3v^2} \cdots\\
&=& t(1+(tv)+(tv)^2+\cdots)(1+t+t^2+\cdots)\cdots\\
&=& \sum_{k=1}^{\infty} (a_0+a_1v+\cdots+a_{k-1}v^{k-1})t^k.
\end{eqnarray*}
for some $a_i\in \N$. Observe, for a fixed $k$ the coefficient of $v^{k-1}$ cannot be zero and there are no higher 
powers of $v$ than $k-1$. Therefore, $Z_{k,k-1}$ has cells of highest dimension $k-1$, 
showing that it is $k-1$ dimensional.
\end{proof}

Now we can prove an analog of the theorem~(\ref{T:cmHilbert}) for the zero fiber $Z_{n,n-1}$ in the smooth 
nested Hilbert scheme of points.

\begin{thm}
The zero fiber $Z_{n,n-1}\subseteq \Hil{n,n-1}$ is a locally complete intersection subscheme of the nonsingular scheme $\Hil{n,n-1}$, therefore it is Cohen Macaulay. 
\end{thm}

\begin{proof}
The subscheme $Z_{n,n-1}$ is equal to the preimage of the zero fiber $Z_n^0$ of the universal 
family under birational map $\rho: \Hil{n,n-1} \rightarrow F_n$ defined before. Note that pull backs of 
the regular functions $\{x,y,p_{rs}:(r,s)\in \mu \setminus (0,0)\}$ vanish on $Z_{n,n-1}$ (see the 
proof of the proposition 2.9 in ~\cite{Hai98}). It is clear that these pullbacks can not be
identically zero on $\Hil{n,n-1}$. Since, by the corollary~\ref{C:dimnested} 
codimension of $Z_{n,n-1}$ is equal to the number of polynomials in $\{x,y,p_{rs}:(r,s)\in \mu \setminus (0,0)\}$, $Z_{n,n-1}$ is a locally complete intersection. As $\Hil{n,n-1}$ is smooth, we see that $Z_{n,n-1}$ is Cohen-Macaulay.
\end{proof}

\section{\textbf{Euler Characteristic formulas a la Haiman}}
\label{S:characterformulas}

In this section we shall deal with Euler characteristic formulas on $\Hil{n,n-1}$. Our approach, as stated in the title of the section is that of Haiman's \cite{Hai02}. In our main theorem of the section, we shall need to use a theorem of Thomason on the equivariant $K$-theory \cite{Thom92}. Assuming that combinatorially oriented reader may not be familiar with the $K$-theory we shall make a brief digression on $K$-theory, state a (simplified) version of the localization theorem of Thomason. Then we shall compute Euler characteristic formulas for the smooth nested Hilbert scheme of points.

\subsubsection{$K$-theory}
In this section we shall briefly review ordinary and equivariant $K$-theory. A basic introduction for equivariant $K$-theory (and more) can be found in the text book \cite{ChrissGinzburg}. 

Let $X$ be a quasiprojective variety. We shall denote the category of coherent $\OO{X}$-modules on $X$ by $Coh(X)$ and the category of locally free $\OO{X}$-modules of finite rank on $X$ by $Vec(X)$.  Clearly $Vec(X) \hookrightarrow Coh(X)$. To single out the elements of $Vec(X)$ from  the elements of $Coh(X)$, we shall use script letters $E\in Vec(X)$ for the locally free sheaves of finite rank, and curly letters $\mathcal{E}\in Coh(X)$ for arbitrary coherent sheaves.

\begin{rem}
Given a vector bundle $\mathfrak{e}$, there exits a locally free sheaf $E$ (the sheaf of sections of $\mathfrak{e}$). Conversely, given a locally free sheaf $E$ of finite rank on the space $X$ one has a vector bundle $\mathfrak{e}=\mathbf{Spec}\  \mathbf{S}(E)$ over $X$, where $\mathbf{S}(E)$ is the symmetric algebra on $E$ (see exercise II.5.18 of \cite{Hartshorne}). Therefore, we have justified our choice of notation 
$Vec(X)$. Until the confusion becomes unbearable, we shall call elements of $Vec(X)$ as vector bundles. 
\end{rem}

We shall abuse the notation for a second and denote either of these categories by $\mathcal{A}=\mathcal{A}_X$. Now, associated with $\mathcal{A}$ is the abelian group $K_0(\mathcal{A})$ generated by the isomorphism classes $[E]$ of objects $E\in Obj(\mathcal{A})$ modulo the subgroup generated by relations $[E_1]-[E_2]+[E_3]$ for every exact sequence 
\begin{equation*}
0 \rightarrow E_1 \rightarrow E_2 \rightarrow E_3 \rightarrow 0.
\end{equation*}
in the category $\mathcal{A}$. 
This quotient group $K_0(\mathcal{A})$ is called as the Grothendieck group of $\mathcal{A}$.

Even though it is a bit confusing, it is customary in the literature to distinguish between $K_0(Coh(X))$ and $K_0(Vec(X))$ by denoting the former by $K_0(X)$ and the latter by $K^0(X)$ and we shall follow this practice.

When $X$ is nonsingular, $K_0(Coh(X))$ and $K^0(Vec(X))$ are isomorphic. To see this, one can solve exercise III.6.9 in \cite{Hartshorne}.

The Grothendieck group $K^0(Vec(X))$ of vector bundles is, in fact, a commutative ring for the tensor product of bundles naturally carries into $K^0(Vec(X))$. Notice that the class of the trivial bundle of rank 1 is the identity element with respect to this multiplication. Note also that tensoring with a locally free sheaf is exact, therefore $K_0(Coh(X))$ has a $K^0(Vec(X))$-module structure via the action
\begin{equation*}
[E]\cdot [\mathcal{F}] = [E \otimes_{\OO{X}} \mathcal{F}]
\end{equation*}

Now assuming $X$ is nonsingular, we shall write $K^0(X)$ for $K^0(Vex(X))$ (equivalently for $K_0(Coh(X))$).  We have the following neat facts, 
\begin{enumerate}
\item The product of the classes $[\mathcal{F}], [\mathcal{G}]\in K_0(X)$  of coherent sheaves is given by
\begin{equation*}
[\mathcal{F}]\cdot [\mathcal{G}] = \sum_i (-1)^i [Tor_i^X(\mathcal{F},\mathcal{G})].
\end{equation*}
\item The dual class $[\mathcal{F}] \spcheck$ of $[\mathcal{F}]\in K_0(X)$ is given by $[\mathcal{F}]\spcheck = \sum_i (-1)^i [Ext^i_X(\mathcal{F},\OO{X})].$
\end{enumerate}

Now, let $G$ be an algebraic group. We shall denote by $Vec^G(X)$ the category of $G$-equivariant algebraic bundles over $X$. Similarly, we have $Coh^G(X)$. As in the ordinary case, $Vec^G(X)$ and $Coh^G(X)$ are abelian categories, and hence the Grothendieck rings $K_0(G,X):=K_0(Coh^G(X))$ and $K^0(G,X):=K^0(Vec^G(X))$ makes sense. 
If $X$ is a one point space, then a $G$-equivariant bundle on $X=pt$ is nothing more than a finite dimensional $G$-module. Therefore, $Vec^G(pt) \cong Coh^G(pt)$ is the category of finite dimensional representations of $G$, hence $K^0(G,pt) \cong K_0(G,pt)$ is the Grothendieck ring of finite dimensional $G$-modules. We shall denote this ring by $Rep(G)$. Multiplication and addition operations of this ring are defined as usual: if
$\tau$ and $\tau'$ are two representations of $G$, then
\begin{equation*}
[\tau]+[\tau'] :=  [\tau \oplus \tau'],\ \text{and}\  [\tau] \cdot [\tau']:=  [\tau \otimes \tau'].
\end{equation*}
By the morphism $X\rightarrow pt$, $K^0(G,X)$ becomes an algebra over the representation ring $Rep(G)$.

For the version of the Thomason's theorem that we state (without proof) here, $G$ shall  
be the $d$-dimensional algebraic torus $\TT^d:=\Spec \C[\mathbf{t},\mathbf{t}^{-1}]$, where $\mathbf{t}$ (and $\mathbf{t}^{-1}$) stands for the variables $t_1,...,t_d$ (for the reciprocals resp.). 

\begin{thm}\label{T:Thomason} (Thomason)
Let $X$ and $Y$ be two separated $\TT^d$-schemes of finite type over $\C$, and let $f:X\rightarrow Y$ be a $\TT^d$-equivariant proper morphism between $X$ and $Y$. Assume that $X$ is nonsingular. For $Z=X$, or $Z=Y$ define
\begin{equation*}
K_0(\TT^d,Z)_{(0)}= \Q(\mathbf{t},\mathbf{t}^{-1}) \otimes_{\Z[\mathbf{t},\mathbf{t}^{-1}]} K_0(\TT^d, Z).
\end{equation*}
Then,
\begin{enumerate}
\item Let $N$ be the conormal bundle of the fixed point locus $X^{\TT^{d}}$ in $X$, and set $\wedge N = \sum_i (-1)^i [\wedge^i N] \in K^0(\TT^d,X^{\TT^{d}})$. Then, $\wedge N$ is invertible in $K^0(\TT^d,X^{\TT^{d}})$.

\item Let $f_{*}:K_0 (\TT^d, X)_{(0)} \rightarrow K_0 (\TT^d, Y)_{(0)}$ be the homomorphism induced by the derived push-forward, that is 
$f_*[M]=\sum_i (-1)^i [R^i f_* M]$, and $f_{*}^{\TT^{d}}:K_0 (\TT^d, 
X^{\TT^{d}})_{(0)} \rightarrow K_0(\TT^d, Y^{\TT^{d}})_{(0)}$ denote the same for the fixed point loci, then 
\begin{equation}\label{E:derpush}
f_* [M]=i_* f^{\TT^d}_* \left( (\wedge N)^{-1}\cdot \sum_k (-1)^k [\Tor^{\shf{X}}_{\C}(\shf{X^{\TT^{d}}},M)] \right),
\end{equation}
where $i_*:K_0(\TT^d, Y^{\TT^{d}})_{(0)} \rightarrow K_0 (\TT^d, Y)_{(0)}$ is induced by 
$i:Y^{\TT^{d}} \hookrightarrow Y$. 
\end{enumerate}
\end{thm}

Now we shall review Haiman's application of this $K$-theoretic localization theorem on the Hilbert scheme of points.

Let $M = \bigoplus  M_{r,s}$ be a finitely-generated doubly graded
module over $\C [\xx ,\yy ]$ or $\C [\xx ,\yy ]^{\mathfrak{S}_n}$.  The {\it
Hilbert series} of $M$ is the Laurent series in two variables
\begin{equation}
\Hs{M}(q,t) = \sum _{r,s} t^{r}q^{s}\dim (M_{r,s}).
\end{equation}
\begin{rem}\label{R:toruscharacter}A doubly graded module $M$ of this form has a natural 
$\TT^2$-module structure, and furthermore value of its character $\tr(M)(\tau)$ at 
$\tau=(t,q)$ is equal to the Hilbert series $\Hs{M}(q,t)$ as described above. 
\end{rem}

Now, suppose $M$ is a $\TT^2$-equivariant coherent sheaf on $\Hil{n}$. Since the Chow morphism $\sigma$ is a 
projective morphism and $S^n \aff$ is affine, the sheaf cohomology modules $H^i(\Hil{n},M)$ are finitely 
generated $\TT^2$-equivariant (which is equivalent to say doubly graded), $\C[\xx,\yy]^{\mathfrak{S}_n}$-modules. 
(see page 228 of \cite{Hartshorne}). The Hilbert series of $M$ is then defined by 
\begin{equation}
\Hs{\Hil{n},M}^i(q,t) = \Hs{H^i(\Hil{n}, M)}(q,t).
\end{equation}
Similarly, for $M$ a $\TT^2$-equivariant coherent sheaf on $\Hil{n,n-1}$, the sheaf cohomology modules $
H^i(\Hil{n,n-1},M)$ are finitely generated $\TT^2$-equivariant $\C[\xx,\yy]^{\mathfrak{S}_n}$-modules. Here, 
we use the fact that the projection $\Hil{n,n-1} \rightarrow \Hil{n}$ is a projective morphism for it 
is restricted from the projective morphism $\Hilb^{n,n-1}_{\PP^2} \rightarrow \Hilb^n_{\PP^2}$, 
and composition of two projective morphism is projective. We define the Hilbert series of $M$ by
\begin{equation}
\Hs{\Hil{n,n-1},M}^i(q,t) = \Hs{H^i(\Hil{n,n-1}, M)}(q,t).
\end{equation}

Recall that for a projective scheme over a field $k$, and a coherent sheaf $\mathcal{F}$ on $X$ 
the Euler characteristic can be defined as the alternating sum \begin{equation}
\chi(\mathcal{F})= \sum_i (-1)^i dim_k H^i(X,\mathcal{F}).
\end{equation}
Similarly we can define the $\TT^2$-equivariant Euler characteristic for $X=\Hil{n}$ and $X=\Hil{n,n-1}$ as the alternating summation
\begin{equation}
\chi_{X,M}(q,t) = \sum_i (-1)^i \Hs{H^i(X,M)}^i(q,t). 
\end{equation}

In \cite{Hai02}, Haiman derives the following form of Atiyah-Bott formula. 

\begin{prop}\cite{Hai02}
Let $M$ be a $\TT^2$-equivariant, locally free coherent sheaf on $\Hil{n}$. Then
\begin{equation}\label{E:AB}
\chi _{\Hil{n},M}(q,t) = \sum _{|\mu | = n} \frac{ \Hs{M(I_{\mu})}(q,t) }{
\prod _{\zeta\in \mu} (1-t^{1+l(\zeta)}q^{-a(\zeta)}) (1-t^{-l(\zeta)}q^{1+a(\zeta)})}
\end{equation}
where $M(I_{\mu})$ is the fiber of $M$ at the torus fixed point $I_{\mu}\in \Hil{n}$. 
\end{prop}

It is important to understand the denominator of the right hand side of the equation. We know that $\TT^2$ fixed points of $\Hil{n}$ are the monomial ideals $I_{\mu}$ indexed by the partitions of $n$, and the cotangent space at the  monomial ideal $I_{\mu}$ has a basis consisting of $\TT^2$ eigenvectors $\{u_{hk},d_{hk}:\ (h,k)\in \mu\}$  with eigenvalues 
\begin{equation*}
(t,q)\cdot d_{hk}=t^{1+l(\zeta)}q^{-a(\zeta)}d_{hk}\ \mbox{and}\ (t,q)\cdot u_{hk} =
t^{-l(\zeta)}q^{1+a(\zeta)}u_{hk},
\end{equation*}
where $a(\zeta)$ and $l(\zeta)$ are the arm and leg of the cell $x=(h,k)\in \mu$.

Along the same lines, to state the version of Atiyah-Bott formula for the nested Hilbert scheme, we shall compute the torus eigenvalues at the fixed points. Using the observation made in the remark \ref{R:eigenvalues}, it is enough to compute eigenvalues for the basis 
$A_1\cup \cdots \cup A_8$ of the tangent space at $(I_{\mu},I_{\mu\setminus \zeta }) \in \Hil{n,n-1}$.

Since $d_{hk}^*(x^{f+1}y^k) = x^hy^g \mod I_{\mu}$, we must have 
\begin{equation}
(t,q)\cdot d_{hk}^* = t^{h-(f+1)}q^{g-k} d_{hk}^* = t^{-1-l(\zeta)}q^{a(\zeta)}d_{hk}^* ,
\end{equation}
and similarly since $u_{hk}^*(x^hy^{g+1}) = x^fy^k \mod I_{\mu}$, we must have 
\begin{equation}
(t,q)\cdot u_{hk}^*= t^{f-h}q^{k-(g+1)} u_{hk}^* = t^{l(\zeta)}q^{-1-a(\zeta)}u_{hk}^*. 
\end{equation}
From now on, even it makes the notation more complex, we need to pay attention which partition 
we are using, for the arm and the leg of the same cell might have different values, depending on 
the monomial ideal we are using. Therefore, we set $a_{\mu}(x)$ (respectively, $l_{\mu}(x)$) to 
be the arm of the cell $x$ (respectively, leg of $x$) in the partition $\mu$. 

Therefore, 

\begin{equation}\label{E:nestedeigenvalues}
\begin{array}{ccc}
  (d_{hk}^*,0)\in A_1 & \Rightarrow & \tau \cdot (d_{hk}^*,0) = t^{-1-l_{\mu}(x)}q^{a_{\mu}(x)} (d_{hk}^*,0), \\
(u_{hk}^*,0)\in A_2 & \Rightarrow & \tau \cdot (u_{hk}^*,0) = t^{l_{\mu}(x)}q^{-1-a_{\mu}(x)}(u_{hk}^*,0), \\
(u^*_{\zeta},0)\in A_3 & \Rightarrow & \tau \cdot (u^*_{\zeta},0)= q^{-1}(u^*_{\zeta},0),\\
(0,d^*_{\zeta})\in A_4 & \Rightarrow & \tau \cdot (0,d^*_{\zeta})= t^{-1}(0,d^*_{\zeta}),\\
(0,d_{hk}^*)\in A_5 & \Rightarrow & \tau \cdot (0,d_{hk}^*)= t^{-1-l_{\mu \setminus \{\zeta\}}(x)}q^{a_{\mu\setminus \zeta}(x)}(0,d_{hk}^*),\\
(0,u_{hk}^*)\in A_6 & \Rightarrow & \tau \cdot (0,u_{hk}^*) = t^{l_{\mu \setminus \zeta}(x)}q^{-1-a_{\mu\setminus \zeta}(x)}  (0,u_{hk}^*),\\
(d_{hk}^*,d_{hk}^*)\in A_7 & \Rightarrow & \tau \cdot (d_{hk}^*,d_{hk}^*)= t^{-1-l_{\mu}(x)}q^{a_{\mu}(x)} (d_{hk}^*,d_{hk}^*),\\
(u_{hk}^*,u_{hk}^*)\in A_8 & \Rightarrow & \tau \cdot (u_{hk}^*,u_{hk}^*) =t^{l_{\mu}(x)}q^{-1-a_{\mu}(x)} (u_{hk}^*,u_{hk}^*).
\end{array}
\end{equation}

Now we can state the Atiyah-Bott formula for the nested Hilbert scheme of points in the plane.

\begin{thm}
Let $M$ be a $\TT^2$-equivariant vector bundle on $\Hil{n,n-1}$, and $\Hs{M(I_{\mu},I_{\nu})}(q,t)$ denote the Hilbert series of the fiber $M(I_{\mu}, I_{\nu})$ of $M$ at the torus fixed point $(I_{\mu},I_{\nu})\in \Hil{n,n-1}$, where $\nu$ is a partition induced from $\mu$ by taking off a corner cell $\zeta \in \mu$ (we write $\nu = \mu \setminus \zeta$). 
We will denote the set of all such pairs of partitions by $\y{n,n-1}$. Then, the equivariant Euler characteristic of $M$ is 

\begin{equation}\label{E:NAB}
\chi _{\Hil{n,n-1},M}(q,t) = \sum _{(\mu,\nu)\in \y{n,n-1}} \frac{\Hs{M(I_{\mu},I_{\nu})}(q,t)}{ (1-t)(1-q)P_1(\mu,\nu)P_2(\mu,\nu)P_3(\mu,\nu)},
\end{equation}
where $\nu=\mu \setminus \zeta$ for some corner cell $\zeta$, and $P_1,P_2$ and $P_3$ are given by
\begin{eqnarray*}
P_1(\mu,\nu) &=& \prod_{\alpha\in \mu \setminus (
\row(\zeta)\cup \col(\zeta))} (1-t^{1+l(\alpha)}q^{-a(\alpha)})(1-t^{-l(\alpha)}q^{1+a(\alpha)}),\\
P_2(\mu,\nu) &=& \prod_{\alpha \in 
\row(\zeta) }(1-t^{1+l(\alpha)}q^{-a(\alpha)})(1-t^{-l(\alpha)}q^{a(\alpha)}),\\
P_3(\mu,\nu) &=& \prod_{\alpha \in \col(\zeta) }(1-t^{-l(\alpha)}q^{1+a(\alpha)})(1-t^{l(\alpha)}q^{-a(\alpha)}).
\end{eqnarray*}
\end{thm}

\begin{proof}
By the remark \ref{R:eigenvalues} above the eigenvalues on the cotangent space must be the reciprocals of the eigenvalues computed in the equations (\ref{E:nestedeigenvalues}). Notice that for a cell $\alpha \in \row(\zeta)$ we have $a_{\mu \setminus \zeta}(\alpha) = a_{\mu}(\alpha)-1$ but $l_{\mu \setminus \zeta }(\alpha) = l_{\mu}(\alpha)$, and similarly, for a cell $\alpha \in \col(\zeta)$ we have $l_{\mu \setminus \zeta }(\alpha) = l_{\mu}(\alpha)-1$ but $a_{\mu \setminus \zeta}(\alpha) = a_{\mu}(\alpha)$. Therefore, combining the data from $A_1$ and $A_5$ we obtain the contribution of $P_2$, similarly combinations of $A_2$ and $A_6$ give $P_3$ 
and combinations of $A_7$ and $A_8$ give $P_1$. Finally, the contribution of $A_3$ and $A_4$ is $(1-t)(1-q)$.

 Now we need to use the localization theorem \ref{T:Thomason} with 
$d=2$. Let $f\colon \Hil{n,n-1}\rightarrow S^{n}(\aff)$ be the projection $\Hil{n} \rightarrow \Hil{n,n-1}$ followed by Chow morphism
$\Hil{n}\rightarrow S^{n}\C ^{2}$.  The group $K_{0}(\TT^2,Y)$ is identified
with the Grothendieck group of finitely-generated doubly graded $\C
[\xx ,\yy ]^{\mathfrak{S}_n}$-modules.  The Hilbert series $\Hs{M}(q,t)$
only depends on the class $[M]\in K_{0}(\TT^2,Y)$ of $M$, and so induces a
$\Z [q,q^{-1},t,t^{-1}]$-linear map
\begin{equation}\label{e:H-on-K0}
\Hs{} : K_{0}(\TT^2,Y)_{(0)}\rightarrow \Q (q,t).
\end{equation}

The fixed-point locus $Y^{\TT^2}$ is a point, so $K_{0}(\TT^2,Y^{\TT^2})_{(0)}= \Q (q,t)$, 
and $\Hs{} \circ i_{*}$ is the identity map on $\Q(q,t)$.  This time, $X^{\TT^2}$ is the finite set 
$\{(I_{\mu},I_{\nu}):\ (\mu,\nu)\in \y{n,n-1} \}$ of pairs of monomial ideals $I_{\nu} \subseteq I_{\mu}$, and $K_{0}(\TT^2,X^{\TT^2})_{(0)}$ is the direct sum of copies of $\Q (q,t)$, one for each pair $(\mu,\nu) \in \y{n,n-1}$. 
Thomason's theorem says that if 
$f_{*}:K_0 (\TT^2,\Hil{n,n-1})_{(0)} \rightarrow K_0 (\TT^2, S^n \aff)_{(0)}$ is the homomorphism induced by the derived push-forward, that is 
$f_*[M]=\sum_i (-1)^i [R^i f_* M]$, and $f_{*}^{\TT^2}:K_0 (\TT^2, (\Hil{n,n-1})^{\TT^2})_{(0)} \rightarrow K_0(\TT^2, {S^n\aff}^{\TT^2})_{(0)}$ denote the same for the fixed point loci, then 
\begin{equation}\label{E:derpush}
f_*[M]=i_*f_*^{\TT^2}\left( (\wedge N)^{-1}\cdot \sum_{k} (-1)^k [Tor^{\shf{\Hil{n,n-1}{}}}_{\C}(\shf{(\Hil{n,n-1})^{\TT^2}},M)] \right),
\end{equation}
where $i_*:K_0(\TT^2, (S^n\aff)^{\TT^2})_{(0)} \rightarrow K_0 (\TT^2, S^n \aff)_{(0)}$ is induced by 
$i:(S^n \aff)^{\TT^2} \hookrightarrow S^n\aff$, and $N$ is the conormal bundle of the fixed point locus $(\Hil{n,n-1})^{\TT^2}$ in $\Hil{n,n-1}$. By part $(1)$ of the theorem \ref{T:Thomason}, $\wedge N$ is invertible in $K^0(\TT^2, (\Hil{n,n-1})^{\TT^2})_{(0)}$. With these identifications $f_{*}^{\TT^2}$ turns into a summation over the torus fixed points which are indexed by $\y{n,n-1}$.

Since for every vector space $V$ and linear endomorphism $\tau$, we have 
\begin{equation*}
\sum_i (-1)^i \tr_{\wedge^i V}(\tau)= det_V(1-\tau),
\end{equation*}
and since the value of the character $\tr(M)(\tau)$ at $\tau=(t,q)$ is equal to the 
Hilbert series $\Hs{M}(q,t)$ (\ref{R:toruscharacter}), applying $\Hs{}$ to both sides of the 
\ref{E:derpush}, we get the Euler characteristic as in the equality \ref{E:NAB} above.
\end{proof}

\section{\textbf{A special case}}\label{S:specialcase}

In \cite{Hai98} Haiman shows that $\Hil{n}$ can be realized as a blow up of $S^n(\aff)$ along a particular subscheme. Let $\shf{}(1)$ be the ample sheaf on $\Hil{n}$ arising from the $\Proj$ construction of the blow up. He further shows that the Euler characteristic of the sheaf $M=\shf{}(m)\otimes \shf{Z_n}$ on $\Hil{n}$ is equal to the $q,t$-Catalan series $C_n(q,t)$. 
In this section, we shall investigate the Atiyah-Bott formula on $\Hil{n,n-1}$ using the pull back sheaf $\eta^*(\shf{}(m)\otimes \shf{Z_n})(q,t)$, where $\eta: \Hil{n,n-1} \rightarrow \Hil{n}$ be the projection onto first component. Let us denote the Euler characteristic of this sheaf by $\mathcal{N}^{(m)}_n(q,t)$. Our conjecture is, as stated in (\ref{T:nestedcharacter}) above, that $\mathcal{N}^{(m)}_n(q,t)$ is a polynomial in $q$ and $t$ with nonnegative integer coefficients. For the convenience of the reader, we have added a table of values of $\mathcal{N}^{(m)}_n(q,t)$ for small $n$ and $m$.

Let $B_n=\pi_*\shf{F_n}$ be the push-down of the structure sheaf of the universal family $F_n$ onto $\Hil{n}$. 
As were shown (proposition 2.12, \cite{Hai98}) $\wedge^n B_n$ is isomorphic to the ample sheaf $\shf{}(1)$. It is also shown that (corollary 2.11, \cite{Hai98}) there exists a sheaf homomorphism 
\begin{equation}
\frac{1}{n}\tr: B_n \rightarrow \shf{\Hil{n}}
\end{equation}
as the left inverse to the canonical inclusion $\shf{\Hil{n}}\hookrightarrow B_n$. Therefore, 
$B_n \cong \shf{\Hil{n}} \oplus B_n'$, where $B'_n$ is the kernel of the trace map. Let $\C_t$ (and $\C_q$)
stand for the one dimensional representation of $\TT^2$ defined by 
$(t,q)\cdot v = tv$ (resp. $(t,q)\cdot v =qv$) and let $\shf{t}$ (and $\shf{q}$) denote the twisted sheaf 
$\shf{t} = \shf{\Hil{n}} \otimes \C_t$ (respectively $\shf{q} = \shf{\Hil{n}}  \otimes \C_q$). Then,

\begin{prop}\cite{Hai98},\cite{Hai02} \label{p:O_zero}
Let $J$ be the sheaf of ideals in $B_n$ generated by $x$, $y$ and $B'_n$.
Then $B_n/J$ is isomorphic as a sheaf of $\shf{H_{n}}$-algebras to
$\shf{Z_n}$ where $Z_n$ is the zero fiber in $\Hil{n}$. Furthermore, we have a $\TT^2$-equivariant locally
$\shf{H_{n}}$-free resolution
\begin{equation}\label{e:res}
0 \rightarrow B_n\otimes \wedge ^{n+1}(B'_n\oplus \shf{t}\oplus \shf{q})\rightarrow \cdots \rightarrow B_n\otimes (B'_n\oplus \shf{t}\oplus \shf{q}) \rightarrow B_n\rightarrow \shf{Z_{n}}\rightarrow 0,
\end{equation}
of $\shf{Z_n}{}$ regarded as an $\shf{\Hil{n}}{}$-module. 
\end{prop}

The Koszul resolution in the above proposition enables us to compute the Hilbert series of the fiber $M(m)(I_{\mu})$ of $M(m):=\shf{}(m)\otimes \shf{Z_n}$ at monomial ideals. 
The following fact is well known. For an exact sequence, the Hilbert series of the initial term is equal to the alternating sum of the Hilbert series of the remaining terms. Therefore, by the remark~\ref{R:toruscharacter} we can calculate $\Hs{\shf{Z_n,I_{\mu}}}(q,t)$ as follows
\begin{eqnarray*}
\Hs{\shf{Z_n,I_{\mu}}}(q,t) &=& \sum_{i=0}^{n+1} (-1)^i \Hs{B_n\otimes \wedge ^{i}(B'_n\oplus \shf{t}\oplus \shf{q})(I_{\mu})}(q,t) \\
&=& \sum_{i=0}^{n+1} (-1)^i \tr_{B_n\otimes \wedge ^{i}(B_n'\oplus \shf{t}\oplus \shf{q})(I_{\mu})}(q,t).
\end{eqnarray*}

Now, tensoring $\shf{Z_n}$ with $\shf{}(m)$ has the effect of multiplying by the torus character 
\begin{equation*}
\prod_{\zeta\in \mu} t^{m\cdot l'(\zeta)}q^{m\cdot a'(\zeta)} = t^{m\cdot n(\mu)}q^{m\cdot n(\mu')},
\end{equation*}
where $n(\mu)= \sum \mu_i(i-1)$. Hence, the Hilbert series of the fiber 
$M(m)(I_{\mu})$  is equal to

\begin{equation}\label{E:numerator}
t^{mn(\mu)}q^{mn(\mu')} \left( \sum_{(h,k)\in \mu} t^hq^k \right) (1-t)(1-q) \prod _{(h,k)\in \mu\setminus \{(0,0)\}} (1-t^{h}q^{k}).
\end{equation}

Putting everything together and using Atiyah-Bott Lefschetz formula, the Euler characteristic of $M(m)$ is equal to 
\begin{equation}\label{E:generalizedCatalan}
\chi_{\Hil{n}, M(m)}(q,t) = \sum _{|\mu | = n} \frac{t^{mn(\mu)}q^{mn(\mu')} ( \sum_{(h,k)\in \mu} t^hq^k ) (1-t)(1-q) \prod _{(h,k)\in \mu\setminus \{(0,0)\}} (1-t^{h}q^{k})}{\prod _{\zeta \in \mu} (1-t^{1+l(\zeta)}q^{-a(\zeta)}) (1-t^{-l(\zeta)}q^{1+a(\zeta)})}
\end{equation}

\begin{rem}
Let us denote this rational expression by $C_n^{(m)}(q,t)$. Recall the definition of the nabla $\nabla$ operator from the introduction. 
It turns out \cite{GaHa96} that $C_n^{(m)}(q,t)=\langle \nabla^m e_n, e_n \rangle$. Note that for $m=1$, $C_n^{(1)}(q,t)$ is equal to the $q,t$-Catalan series that we have defined earlier in detail. 
We should mention here that for $m\geq 2$, there are conjectural combinatorial formulations similar to that of $m=1$ (see \cite{hhlru05}).
\end{rem}

Now, we turn back to the nested Hilbert schemes. Since the fiber of the locally free sheaf $\eta^*(M(m))=\eta^*(\shf{}(m)\otimes \shf{Z_n})$ at $(I_{\mu},I_{\nu})$ is the same as the fiber at $I_{\mu}$ of $M(m)$, the Hilbert series of the fiber $\eta^*(M(m))(I_{\mu},I_{\nu})$ is the same as (\ref{E:numerator}) above. Therefore, we obtain

\begin{prop}\label{T:nestedCharacter} Let $\eta:\Hil{n,n-1} \rightarrow \Hil{n}$ be the projection. 
Then the Euler characteristic $\chi _{\Hil{n,n-1},\eta^*(\shf{}(m)\otimes \shf{Z_n})}(q,t)$ is 
equal to the rational expression 
\begin{equation}\label{E:nestedCatalan}
\sum _{(\mu,\nu)\in \y{n,n-1} } \frac{t^{mn(\mu)}q^{mn(\mu')} (\sum_{(h,k)\in \mu} t^hq^k) (1-t)(1-q) \prod _{(h,k)\in \mu\setminus (0,0)} (1-t^{h}q^{k})}{
(1-t)(1-q)P_1(\mu,\nu)P_2(\mu,\nu)P_3(\mu,\nu)},
\end{equation}
where the sum is over all nested pairs $(\mu,\nu)\in \y{n,n-1}$, and if $\zeta$ is the cell of $\mu$ that is not contained in $\nu$, then the terms $P_1,P_2$ and $P_3$ appearing 
in denominator are given by
\begin{eqnarray*}
P_1(\mu,\nu) &=& \prod_{\zeta \in \mu \setminus (
\row(\zeta)\cup \col(\zeta))} (1-t^{1+l(\zeta)}q^{-a(\zeta)})(1-t^{-l(\zeta)}q^{1+a(\zeta)})\\
P_2(\mu,\nu) &=& \prod_{\zeta \in 
\row(\zeta) }(1-t^{1+l(\zeta)}q^{-a(\zeta)})(1-t^{-l(\zeta)}q^{a(\zeta)})\\
P_3(\mu,\nu) &=& \prod_{\zeta \in \col(\zeta) }(1-t^{-l(\zeta)}q^{1+a(\zeta)})(1-t^{l(\zeta)}q^{-a(\zeta)}).
\end{eqnarray*}  
\end{prop}

We now (re)state our conjecture that the Euler characteristic 
\begin{equation}
\mathcal{N}^{(m)}_n(q,t) := \chi _{\Hil{n,n-1},M(m)}(q,t)
\end{equation}
of the pull back of the sheaf $M(m)$ is a polynomial in $q$ and $t$ with nonnegative coefficients. We further conjecture that $\mathcal{N}^{(1)}_n(q,t)$ is equal to the combinatorial nested $q,t$-Catalan series $N_n(q,t)$ that we have defined in \ref{D:nestedcatalan}.

From now on we would like to call $\mathcal{N}^{(m)}_n(q,t)$ as the (generalized) \textsl{nested $q,t$-Catalan series}. 
Looking at the small values of $n$, we also conjecture that 
\begin{equation}\label{E:ncatalans}
\mathcal{N}^{(1)}_n(1,1)=N_n(1,1)=\frac{n}{2} C_{n+1}^{(1)}(1,1)=\frac{n}{2(n+1)} {2(n+1) \choose n+1}.
\end{equation}

Maple experiments further indicates that the nested Catalan series has properties similar to those of $C^{(m)}_n(q,t)$, i.e., the bivariate symmetry $\mathcal{N}^{(m)}_n(q,t)=\mathcal{N}^{(m)}_n(t,q)$. For these reasons and others, we believe that the nested $q,t$-Catalan series deserves further investigations.

\section{\textbf{Final remarks}}
The following connections of the $\TT^2$-weights on the (co)tangent spaces of the smooth nested Hilbert schemes with Pieri-Macdonald coeffients have been pointed to the author by Mark Haiman. 

We use the notation as before and denote by $\nu\in \y{m}-1$ the partition obtained from $\mu$ by taking off a corner cell $\zeta \in \mu$. In ~\cite{GaHa98}, using the Pieri rules (VI.6 of ~\cite{Mac95}) for the \textit{original} Macdonald polynomials, Garsia and Haiman show that the coefficient of the modified Macdonald polynomial $\widetilde{H_{\mu}}(z;q,t)$ in the symmetric function $e_1(z)\widetilde{H_{\nu}}(z;q,t)$ is given, in our notation, by 
\begin{equation*}
d_{\mu \nu}(q,t) = \prod_{\alpha \in 
\row(\zeta)} \frac{q^{a(\alpha)-1}-t^{l(\alpha)+1}}{q^{a(\alpha)}-t^{l(\alpha)+1}} \prod_{\alpha \in \col(\zeta)} \frac{t^{l(\alpha)-1}-q^{a(\alpha)+1}}{t^{l(\alpha)}-q^{a(\alpha)+1}}.
\end{equation*}
Similarly, for the dual $\partial_{p_1}$ of the multiplication by $e_1$, it is shown in the same paper that the coefficient of the modified Macdonald polynomial $\widetilde{H_{\nu}}(z;q,t)$ in $\partial_{p_1} \widetilde{H_{\mu}}(z;q,t)$ can be written, in our notation, as
\begin{equation*}
c_{\mu \nu}(q,t) =  \prod_{\alpha \in 
\row(\zeta)} \frac{t^{l(\alpha)}-q^{a(\alpha)+1}}{t^{l(\alpha)}-q^{a(\alpha)}} \prod_{\alpha \in \col(\zeta)} \frac{q^{a(\alpha)}-t^{l(\alpha)+1}}{q^{a(\alpha)}-t^{l(\alpha)}}.
\end{equation*}

Let us denote by $\Pi_{\mu,\nu}$ the product $(1-t)(1-q)P_1(\mu,\nu)P_2(\mu,\nu)P_3(\mu,\nu)$ which comes from the $\TT^2$-weights on the cotangent space at the torus fixed point $(I_{\mu},I_{\nu})$ of $\Hil{n,n-1}$. Likewise, let us denote by $\Pi_{\mu}$ and $\Pi_{\nu}$ the products  
\begin{equation*}
\prod_{\alpha \in \mu} (1-t^{1+l(\alpha)}q^{-a(\alpha)}) (1-t^{-l(\alpha)}q^{1+a(\alpha)})\ \text{and}\ 
\prod_{\alpha\in \nu} (1-t^{1+l_{\nu}(\alpha)}q^{-a_{\nu}(\alpha)}) (1-t^{-l_{\nu}(\alpha)}q^{1+a_{\nu}(\alpha)})
\end{equation*}
coming from the $\TT^2$-action on the cotangent spaces at the torus fixed points $I_{\mu}$ of $\Hil{n}$ and $I_{\nu}$ of $H_{n-1}$, respectively.
\textit{Caution:} we are using subscripts in $a_{\nu}(x)$ and $l_{\nu}(x)$ to distinguish the arm and the leg of $x\in \nu \subseteq \mu$ from the arm and leg computed in the partition $\mu$. 
Then, after making the necessary cancellations we see that 
\begin{equation}\label{E:MacPieri}
c_{\mu \nu}= \Pi_{\mu} / \Pi_{\mu,\nu}\  \mbox{and}\  d_{\mu \nu}= t^{l'(\zeta)}q^{a'(\zeta)} \Pi_{\nu}/(1-t)(1-q) \Pi_{\mu,\nu}.
\end{equation}
 
It is also pointed to us by Mark Haiman that there is yet another way to derive the equalities of ~\ref{E:MacPieri} in a more conceptual way using the \textit{isospectral} Hilbert scheme of points, however, we are not going to record it here.

Finally, we would like to mention that the author and J. Haglund have conjectural symmetric function interpretations of the nested $q,t$-Catalan series. We also have a generalization of the $q,t$-Catalan series with four parameters. We shall write these in a forth coming paper.

\textbf{Acknowledgements.} This work has been done while the author was a graduate student. The author would like to thank to his advisor J. Haglund for, among many other things, letting him work on the subject that the author found interesting. The author also would like to thank M. Haiman for brief but extremely valuable conversations.

\section{\textbf{Tables}}\label{S:tables}

Here are the few nested $q,t$-Catalan series. 

\begin{eqnarray*}
\mathcal{N}^{(1)}_1(q,t)=t+q+t^2+tq+q^2
\end{eqnarray*}

\begin{eqnarray*}
\mathcal{N}^{(1)}_2(q,t) &=&
{t}^{5}+{t}^{4}+ {t}^{4}q+{t}^{3}+2q{t}^{3}+{q}^{2}{t}^{3}+2q{t}^{2}+2{q}^{2}{t}^{2}+ {q}^{3}{t}^{2}+qt+2{q}^{2}t+2{q}^{3}t+{q}^{4}t+\\ && {q}^{4}+{q}^{3}+{q}^{5}
\end{eqnarray*}

\begin{eqnarray*}
\mathcal{N}^{(1)}_3(q,t) &=& {q}^{8}+4{q}^{2}{t}^{4}+4{q}^{4}{t}^{3}+2{t}^{4}q+2{q}^{7}t+ {t}^{8}+{t}^{3}{q}^{6}+ {q}^{6}+{q}^{8}t+ 
4{q}^{2}{t}^{5}+ 2{q}^{3}{t}^{5}+ {q}^{4}{t}^{5}+ \\ && {q}^{3}{t}^{6}+
{q}^{5}{t}^{4}+2{q}^{2}{t}^{6}+3{t}^{6}q+3{t}^{5}q+{q}^{7}{t}^{2} 
+{t}^{7}+ 4{q}^{3}{t}^{4}+2{q}^{6}{t}^{2}+2{q}^{4}{t}^{4}+ {t}^{6}+{q}^{7}+
\\ && {t}^{9}+2{q}^{4}t+{q}^{9}+3{q}^{5}t+{t}^{8}q+{t}^{7}{q}^{2}+ 
4{q}^{5}{t}^{2}+{q}^{2}{t}^{2}+q{t}^{3}+{q}^{3}t+3{q}^{2}{t}^{3}+ 
5{q}^{3}{t}^{3}+\\ && 3{q}^{3}{t}^{2}+3t{q}^{6}+4{q}^{4}{t}^{2}+ 
2{t}^{3}{q}^{5}+2q{t}^{7}
\end{eqnarray*}

\begin{eqnarray*}
\mathcal{N}^{(1)}_4(q,t) &=& {q}^{10}{t}^{4}+2{q}^{12}t+2{q}^{4}{t}^{9}+
4{q}^{3}{t}^{9}+{q}^{2}{t}^{4}+4{q}^{4}{t}^{3}+4{t}^{7}{q}^{5}+{t}^{12}+
{t}^{13}+ 2{q}^{7}t+9{t}^{3}{q}^{6}+\\ && 4{q}^{6}{t}^{6}+2{q}^{11}{t}^{2}+4{q}^{9}{t}^{3}+4{q}^{10}t+4{q}^{7}{t}^{5}+2{q}^{8}{t}^{5}+4{q}^{2}{t
}^{10}+ 4{q}^{8}{t}^{4}+3{q}^{11}t+{q}^{8}{t}^{6}+\\ && {q}^{11}{t}^{3}+{
q}^{9}{t}^{5}+{q}^{13}t+{t}^{10}+3{t}^{11}q+4{t}^{10}q+9{t}^{3}{
q}^{7}+{q}^{12}{t}^{2}+ 9{t}^{7}{q}^{3}+7{t}^{7}{q}^{4}+4{t}^{8}{
q}^{4}+ \\ && 2{q}^{10}{t}^{3}+ 2{q}^{9}{t}^{4}+3{q}^{8}t+7{q}^{2}{t}^
{8}+7{q}^{3}{t}^{8}+2{q}^{2}{t}^{5}+7{q}^{3}{t}^{5}+ 10{t}^{5}{
q}^{5}+ 10{q}^{4}{t}^{5}+ 7{q}^{5}{t}^{6}+ \\ && 10{q}^{4}{t}^{6}+ 7{q}^
{6}{t}^{5}+ 9{q}^{3}{t}^{6}+10{q}^{5}{t}^{4}+5{q}^{2}{t}^{6}+
{t}^{6}q+2q{t}^{12}+{q}^{2}{t}^{12}+{t}^{8}{q}^{6}+{q}^{13}+\\ &&
6{q}^{7}{t}^{2}+ 4{q}^{3}{t}^{4}+5{q}^{6}{t}^{2}+ 8{q}^{4}{t}^{4}+{q}^{12}+{q}^{14}+{q}^{11}+ 2{t}^{8}{q}^{5}+2{t}^{7}{q}^{6}+{t}^{7}{q}^{7}+
3{t}^{8}q+ \\ && 6{t}^{7}{q}^{2}+2{q}^{5}{t}^{2}+7{q}^{8}{t}^{2}+ {q}^{
3}{t}^{3}+2{t}^{6}{q}^{7}+ t{q}^{6}+{q}^{4}{t}^{2}+7{t}^{3}{q}^{5}+
6{t}^{2}{q}^{9}+10{t}^{4}{q}^{6}+2q{t}^{7}+\\ && 4q{t}^{9}+ {t}^{11}+
{t}^{14}+ 4{q}^{9}t+6{q}^{2}{t}^{9}+{q}^{10}+7{q}^{8}{t}^{3}+
7{q}^{7}{t}^{4}+ {t}^{9}{q}^{5}+ q{t}^{13}+2{q}^{2}{t}^{11}+ \\ && 
{q}^{3}{t}^{11}+{q}^{4}{t}^{10}+2{q}^{3}{t}^{10}+4{q}^{10}{t}^{2}
\end{eqnarray*}

\begin{eqnarray*}
\mathcal{N}^{(2)}_2(q,t) &=& {t}^{3}+{t}^{2}+q{t}^{2}+qt+{q}^{2}t+{q}^{2}+{q}^{3}
\end{eqnarray*}

\begin{eqnarray*}
\mathcal{N}^{(2)}_3(q,t) &=& 
{q}^{8}+3{q}^{2}{t}^{4}+2{q}^{4}{t}^{3}+{t}^{4}q+{q}^{7}t+{t}^{8}+
{q}^{6}+2{q}^{2}{t}^{5}+{q}^{3}{t}^{5}+{q}^{2}{t}^{6}+2{t}^{6}q+2
{t}^{5}q+\\ && {t}^{7}+ 2{q}^{3}{t}^{4}+ {q}^{6}{t}^{2}+{q}^{4}{t}^{4}+{t}
^{6}+{q}^{7}+{q}^{4}t+2{q}^{5}t+2{q}^{5}{t}^{2}+{q}^{2}{t}^{2}+ 
2{q}^{2}{t}^{3}+3{q}^{3}{t}^{3}+ \\ && 2{q}^{3}{t}^{2}+2t{q}^{6}+3{q}^
{4}{t}^{2}+{t}^{3}{q}^{5}+q{t}^{7}
\end{eqnarray*}

\begin{eqnarray*}
\mathcal{N}^{(2)}_4(q,t) &=& 
4{q}^{4}{t}^{9}+6{q}^{3}{t}^{9}+6{t}^{7}{q}^{5}+4{q}^{8}{t}^{5
}+7{q}^{8}{t}^{3}+{q}^{5}{t}^{10}+{q}^{4}{t}^{11}+4{t}^{3}{q}^{6}+ 
6{q}^{6}{t}^{6}+ 4{q}^{9}{t}^{4}+\\ && 8{q}^{7}{t}^{4}+{q}^{13}+6{q}^
{8}{t}^{4}+2{t}^{4}{q}^{10}+4{q}^{10}{t}^{3}+2{q}^{9}{t}^{5}+{t}
^{14}q+2{t}^{13}q+{t}^{9}{q}^{6}+3{q}^{12}t+\\ && 2{q}^{11}{t}^{3}+{q}
^{9}{t}^{6}+{q}^{11}{t}^{4}+ 6{q}^{7}{t}^{5}+2{q}^{8}{t}^{6}+6{t}
^{3}{q}^{7}+2{q}^{13}t+6{t}^{7}{q}^{3}+8{t}^{7}{q}^{4}+{t}^{13}{
q}^{2}+\\ && 6{t}^{8}{q}^{4}+4{q}^{2}{t}^{8}+7{q}^{3}{t}^{8}+{q}^{14}+
{q}^{12}{t}^{3}+{q}^{14}t+{q}^{3}{t}^{5}+8{t}^{5}{q}^{5}+5{q}^{4}{
t}^{5}+8{q}^{5}{t}^{6}+8{q}^{4}{t}^{6}+\\ && 5{q}^{9}{t}^{2}+8{q}^{6
}{t}^{5}+4{q}^{3}{t}^{6}+5{q}^{5}{t}^{4}+{q}^{2}{t}^{6}+2{t}^{8}
{q}^{6}+2{q}^{7}{t}^{2}+{t}^{12}+{t}^{13}+{q}^{6}{t}^{2}+2{q}^{4}{
t}^{4}+\\ && 2{t}^{10}q+5{t}^{10}{q}^{2}+{q}^{12}+4{t}^{8}{q}^{5}+4{
t}^{7}{q}^{6}+2{t}^{7}{q}^{7}+{q}^{15}+{t}^{7}{q}^{8}+2{t}^{7}{q}^
{2}+4{q}^{8}{t}^{2}+4{t}^{6}{q}^{7}+\\ && 5{t}^{2}{q}^{10}+{q}^{13}{t}
^{2}+{t}^{3}{q}^{5}+3{t}^{11}q+{t}^{5}{q}^{10}+2{q}^{12}{t}^{2}+4
{q}^{11}{t}^{2}+4{t}^{10}{q}^{3}+2{t}^{10}{q}^{4}+4{q}^{2}{t}^
{11}+\\ && {t}^{14}+8{t}^{4}{q}^{6}+{t}^{8}{q}^{7}+{q}^{3}{t}^{12}+q{t}^{9
}+{q}^{9}t+5{q}^{2}{t}^{9}+3{t}^{12}q+2{t}^{11}{q}^{3}+2{q}^{2
}{t}^{12}+2{t}^{9}{q}^{5}+\\ && {t}^{15}+6{q}^{9}{t}^{3}+2{q}^{10}t+3
{q}^{11}t
\end{eqnarray*}

\begin{eqnarray*}
\mathcal{N}^{(3)}_3(q,t) &=& 
2{q}^{4}{t}^{3}+{q}^{7}t+{q}^{8}{t}^{3}+3{t}^{3}{q}^{6}+{q}^{7}{t}
^{4}+2{t}^{3}{q}^{7}+2{t}^{7}{q}^{3}+{t}^{7}{q}^{4}+2{q}^{8}t+
2{q}^{2}{t}^{8}+{q}^{3}{t}^{8}+\\ && {q}^{2}{t}^{5}+3{q}^{3}{t}^{5}+2{t
}^{5}{q}^{5}+3{q}^{4}{t}^{5}+{q}^{5}{t}^{6}+2{q}^{4}{t}^{6}+{q}^{9
}{t}^{2}+{q}^{6}{t}^{5}+3{q}^{3}{t}^{6}+3{q}^{5}{t}^{4}+ \\ && 
2{q}^{2}{t}^{6}+3{q}^{7}{t}^{2}+2{q}^{3}{t}^{4}+2{q}^{6}{t}^{2}+3{q}^{
4}{t}^{4}+{t}^{9}+{t}^{10}q+{q}^{10}+{q}^{11}+{q}^{9}+2{t}^{8}q+\\ &&
3{t}^{7}{q}^{2}+{q}^{5}{t}^{2}+2{q}^{8}{t}^{2}+{q}^{3}{t}^{3}+3{t}^{
3}{q}^{5}+2{t}^{4}{q}^{6}+q{t}^{7}+2q{t}^{9}+2{q}^{9}t+{q}^{2}{t
}^{9}+{t}^{10}+\\ && {t}^{11}+{q}^{10}t
\end{eqnarray*}

\begin{eqnarray*}
\mathcal{N}^{(3)}_4(q,t) &=& 
6{t}^{9}{q}^{9}+9{q}^{4}{t}^{12}+6{q}^{3}{t}^{12}+11{t}^{10}{q
}^{5}+{q}^{20}+4{q}^{4}{t}^{9}+{q}^{3}{t}^{9}+2{t}^{7}{q}^{5}+2{
q}^{6}{t}^{6}+2{q}^{10}{t}^{3}+\\ && {q}^{8}{t}^{4}+{q}^{12}{t}^{2}+{q}^{9
}{t}^{3}+4{q}^{11}{t}^{3}+6{q}^{12}{t}^{3}+4{q}^{9}{t}^{4}+12{
t}^{9}{q}^{6}+2{q}^{6}{t}^{14}+4{q}^{5}{t}^{14}+{q}^{9}{t}^{12}+\\ &&
10{t}^{6}{q}^{10}+6{t}^{5}{q}^{13}+10{t}^{5}{q}^{9}+10{q}^{9}{t}
^{7}+11{q}^{8}{t}^{6}+7{q}^{13}{t}^{3}+{q}^{15}t+7{q}^{14}{t}^{3
}+9{q}^{12}{t}^{4}+\\ && 10{q}^{11}{t}^{5}+8{q}^{12}{t}^{5}+7{q}^{10
}{t}^{4}+12{q}^{9}{t}^{6}+2{t}^{5}{q}^{7}+8{t}^{9}{q}^{8}+10{t
}^{8}{q}^{8}+2t{q}^{16}+{t}^{8}{q}^{4}+\\ && 
{t}^{15}q+5{t}^{15}{q}^{2}+ 2{q}^{11}{t}^{9}+8{t}^{12}{q}^{5}+{q}^{8}{t}^{13}+2{q}^{7}{t}^{13}+8{t}^{4}{q}^{13}+4{q}^{12}{t}^{7}+{q}^{13}{t}^{8}+\\ &&
6{q}^{6}{t}^{12}+{q}^{15}{t}^{6}+4{q}^{6}{t}^{13}+4{q}^{8}{t}^{11}+6{q}^{7
}{t}^{11}+{q}^{10}{t}^{11}+6{q}^{5}{t}^{13}+2{q}^{9}{t}^{11}+2{q
}^{10}{t}^{10}+\\ && {q}^{11}{t}^{10}+11{t}^{8}{q}^{6}+3{q}^{18}t+4{q}
^{16}{t}^{3}+6{q}^{12}{t}^{6}+2{q}^{17}{t}^{3}+2{q}^{19}t+{q}^{
18}{t}^{3}+2{q}^{16}{t}^{4}+\\ && 2{q}^{15}{t}^{5}+{q}^{16}{t}^{5}+6{q
}^{14}{t}^{4}+4{q}^{13}{t}^{6}+6{q}^{8}{t}^{5}+9{q}^{11}{t}^{4}+
{q}^{12}{t}^{9}+{q}^{20}t+6{t}^{8}{q}^{5}+7{t}^{7}{q}^{6}+\\ && 12{t}^
{7}{q}^{7}+12{t}^{7}{q}^{8}+{t}^{16}{q}^{5}+7{t}^{6}{q}^{7}+{t}^{
18}+{t}^{19}+4{t}^{14}{q}^{2}+4{q}^{4}{t}^{15}+2{q}^{4}{t}^{16}+
q{t}^{20}+\\ && 2q{t}^{19}+8{t}^{8}{q}^{9}+{q}^{19}+{q}^{18}+{q}^{17}{t}
^{4}+11{q}^{10}{t}^{5}+4{q}^{15}{t}^{4}+8{q}^{10}{t}^{7}+4{q}^
{17}{t}^{2}+{q}^{19}{t}^{2}+\\ && 2{q}^{18}{t}^{2}+2{t}^{10}{q}^{3}+7{
t}^{10}{q}^{4}+{t}^{12}{q}^{2}+12{t}^{8}{q}^{7}+8{q}^{11}{t}^{6}+2
{t}^{13}{q}^{2}+4{t}^{9}{q}^{10}+2{q}^{14}{t}^{6}+\\ && 4{q}^{11}{t}
^{8}+2{q}^{12}{t}^{8}+9{t}^{11}{q}^{4}+2{t}^{7}{q}^{13}+4{q}^{
3}{t}^{11}+6{q}^{4}{t}^{14}+{t}^{20}+4{t}^{16}{q}^{3}+2{t}^{15}{
q}^{5}+{t}^{15}{q}^{6}+\\ && {t}^{14}{q}^{7}+8{t}^{11}{q}^{6}+{q}^{14}{t}^
{7}+6{q}^{11}{t}^{7}+10{t}^{11}{q}^{5}+10{t}^{10}{q}^{6}+8{t}^
{10}{q}^{7}+6{t}^{10}{q}^{8}+6{q}^{10}{t}^{8}+\\ && 3{t}^{17}q+2{t}^
{16}q+10{t}^{9}{q}^{7}+4{q}^{14}{t}^{5}+2{q}^{8}{t}^{12}+4{q}^
{7}{t}^{12}+4{q}^{9}{t}^{10}+3{q}^{17}t+4{q}^{2}{t}^{17}+\\ && 5{q}^
{2}{t}^{16}+7{q}^{3}{t}^{14}+{q}^{21}+10{t}^{9}{q}^{5}+2{t}^{2}{
q}^{13}+5{q}^{15}{t}^{2}+4{q}^{14}{t}^{2}+5{q}^{16}{t}^{2}+{t}^{
21}+3{t}^{18}q+\\ && 2{t}^{18}{q}^{2}+6{t}^{15}{q}^{3}+8{q}^{4}{t}^{
13}+7{q}^{3}{t}^{13}+{t}^{19}{q}^{2}+{t}^{18}{q}^{3}+2{t}^{17}{q}^
{3}+{t}^{17}{q}^{4}+6{t}^{3}{q}^{15}
\end{eqnarray*}

\bibliographystyle{alpha}

\end{document}